\documentclass[a4paper,draft]{amsart}
\usepackage{amsthm}
\usepackage{amscd}
\usepackage[final]{graphicx}
\usepackage{amssymb,amsmath}  %
\usepackage{latexsym} %
\usepackage[english]{babel}
\usepackage{tabularx}
\usepackage{xypic}
\xyoption{curve}
\usepackage{longtable}

\newtheorem{lem}{Lemma}%
    \newtheorem{prop}[lem]{Proposition}
    \newtheorem{thm}[lem]{Theorem}
    \newtheorem{cor}[lem]{Corollary}
    \newtheorem*{qs}{Question} 
\theoremstyle{remark}
    \newtheorem{rem}[lem]{Remark}

\DeclareMathOperator{\Sym}{Sym}
\DeclareMathOperator{\Aut}{Aut}

\DeclareMathOperator{\rank}{rank}
\DeclareMathOperator{\Hom}{Hom}

\newcommand{\Sp}{\mathop{\mathrm {Sp}}\nolimits}
\newcommand{\GL}{\mathop{\mathrm {GL}}\nolimits}
\newcommand{\SL}{\mathop{\mathrm {SL}}\nolimits}
\def\s{\mathfrak S}

\newcommand{\op}[1]{{\operatorname{#1}}}
\def\cA{\mathcal A}
\def\cO{\mathcal O}
\def\cX{\mathcal X}
\def\cP{\mathcal P}

\newcommand{\V}{\bb V}

\newcommand{\ab}[1]{\cA_{{#1}}}
\newcommand{\xb}[1]{\cX_{{#1}}}
\newcommand{\Vor}{{\op{Vor}}}
\newcommand{\Perf}{{\op{perf}}}
\newcommand{\Sat}{{\op{Sat}}}
\newcommand{\Igu}{{\op{Igu}}}
\newcommand{\tor}{{\op{tor}}}
\newcommand{\AVOR}[1]{{\cA_{#1}^{\Vor}}}
\newcommand{\APERF}[1]{{\cA_{#1}^{\Perf}}}
\newcommand{\AIGU}[1]{{\cA_{#1}^{\Igu}}}
\newcommand{\ATOR}[1]{{\cA_{#1}^{\tor}}}
\newcommand{\XVOR}[1]{{\cX_{#1}^{\Vor}}}

\newcommand{\ASAT}[1]{{\cA_{#1}^{\Sat}}}
\newcommand{\M}[1]{\mathcal M_{{#1}}}
\newcommand{\Mm}[2]{\mathcal M_{{#1},{#2}}}
\newcommand{\Mb}[1]{\overline{\mathcal M}_{{#1}}}

\newcommand{\coh}[3][\Q]{H^{#3}(#2;#1)}
\newcommand{\cohc}[3][\Q]{H_c^{#3}(#2;#1)}

\newcommand{\Ll}{\mathbf L}

\newcommand{\eul}[1]{\mathbf{e}_{\op{Hdg}}({#1})}
\newcommand{\sing}{{\op{sing}}}

\def\bb{\mathbb}
\def\la{\langle}
\def\ra{\rangle}
\def\phi{\varphi}

\def\Z{\bb{Z}}
\def\Q{\bb{Q}}
\def\R{\bb{R}}
\def\C{\bb{C}}

\def\H{\bb{H}}

\newcommand{\Pp}[1]{\bb{P}^{#1}}

\def\co{\colon\thinspace}
\def\pu{\bullet}
\def\mil{\mspace{-9mu}}

\numberwithin{equation}{section}

\begin{document}

\author{Klaus Hulek}
\address{Leibniz Universit\"at Hannover, Institut f\"ur Algebraische Geometrie, Wel\-fen\-gar\-ten~1, D-30167 Hannover, Germany}
\email{hulek@math.uni-hannover.de}
\author{Orsola Tommasi}
\address{Leibniz Universit\"at Hannover, Institut f\"ur Algebraische Geometrie, Wel\-fen\-gar\-ten~1, D-30167 Hannover, Germany}
\email{tommasi@math.uni-hannover.de}

\title[Cohomology of the second Voronoi compactification of $\mathcal A_4$]{Cohomology of the second Voronoi compactification of $\mathcal A_4$}
\date{January 2012}

\begin{abstract}
In this paper we compute the cohomology groups of the second Voronoi and the perfect cone compactification $\AVOR 4$ 
and $\APERF 4$ respectively, of the moduli space of abelian fourfolds in degree $\leq 9$. 
The main tool is the investigation of the strata of $\AVOR 4$ corresponding to semi-abelic varieties with constant torus rank. 
\end{abstract}
\maketitle

\section{Introduction and plan}

The moduli space $\ab g$ of principally polarized abelian varieties of genus $g$ are much studied objects 
in algebraic geometry. Although much progress has been made in
understanding the geometry of these spaces, we still know relatively little about the cohomology or the Chow
groups of $\ab g$ and its compactifications. 
These are difficult questions even for low genus.
Mumford in his seminal paper \cite{Mu-curves} computed the Chow ring of $\overline{\mathcal M}_2$, or what is the same,
of the second Voronoi compactification $\AVOR 2$. It was also in this paper that he laid the foundations 
for the study of the Chow ring of $\overline{\mathcal M}_g$ in general. 
Lee and Weintraub \cite{LW} have investigated the cohomology of certain level covers of $\AVOR 2$. 
The cohomology of $\ab 3$ and of the Satake compactification $\ASAT 3$ were determined by Hain \cite{Hain}, while
the Chow group of the second Voronoi compactification 
$\AVOR 3$ had earlier been computed by van der Geer \cite{vdG}. The authors of this paper proved in \cite{HT}
that the Chow ring and the cohomology ring of $\AVOR g$ are isomorphic for $g=2,3$. 

Very little is known about the topology of $\ab g$ and its compactifications
in general. A positive exception is given by the subring generated by the Chern classes $\lambda_i$ of the Hodge bundle in the Chow ring or the cohomology ring of $\ab g$. By \cite{vdG-lambda} this subring is known explicitly; in particular, it is generated by the odd degree Hodge classes $\lambda_{2i+1}$. 
Furthermore, by a classical result by Borel \cite{borel} the Hodge classes $\lambda_{2i+1}$ generate the stable cohomology of $\ab g$, which is defined in terms of the natural maps $\ab {g'}\hookrightarrow \ab g$ for $g'<g$ given by multiplication with a fixed abelian variety of dimension $g-g'$. 
Note that by the construction of  $\ab g$ as a quotient of Siegel space, the rational cohomology of $\ab g$ coincides with the group cohomology of the symplectic group $\Sp(2g,\Z)$. 

In this paper we investigate the case of genus $4$ of whose cohomology very little is known. 
There are two reasons why we believe that it is worth the effort to undertake this study. One is that
the spaces $\ab 2$ and $\ab 3$ are very close to the moduli spaces $\M 2$ and $\M 3$ whose cohomology is 
rather well understood, whereas $\ab 4$ is the first Siegel moduli space where the Torelli map fails 
to be birational and thus one might expect new phenomena. The second reason is
that, as our discussion of the boundary strata shows, the complexity of the structure of $\ab g$
and its compactifications grows dramatically with $g$. At the moment calculations on the cohomology of $\ab 5$, 
or any of its compactifications, seems far out of reach and thus $\ab 4$ is the only remaining low genus case 
where the computation of the entire cohomology currently seems within reach. 

In this paper we thus investigate the cohomology of toroidal compactifications of $\ab 4$. %
 In general there are several meaningful compactifications of $\ab g$. Besides the second Voronoi
compactification $\AVOR g$ there is the perfect compactification $\APERF g$, given by the second Voronoi
decomposition and the perfect cone (or first Voronoi) decomposition respectively, as well as the
Igusa compactification $\AIGU g$. It was shown by Alexeev \cite{Al} and Olsson \cite{Ol}
that (at least up to normalization) $\AVOR g$ represents a geometric functor given by stable semi-abelic varieties. 
On the other hand 
$\APERF g$ is, as was proved by Shepherd-Barron \cite{S-B}, a canonical model in the sense of Mori theory, i.e.
its canonical bundle is ample, if $g\geq 12$. Finally, Igusa constructed $\AIGU g$ as a partial blow-up of
$\ASAT g$ and it was shown by Namikawa \cite{Nam} that Igusa's model is the toroidal compactification defined by the
central cone decomposition. 
In genus $g \leq 3$ all of the above toroidal compactifications coincide. 
In genus $4$ the Igusa and the perfect cone decomposition coincide and the second Voronoi compactification $\AVOR 4$ 
is a blow-up of $\APERF 4$. However, for general $g$ all three compactifications are different. 

The main result of 
our paper is the determination  of the 
Betti numbers of $\APERF 4$ of degree less than or equal to $9$ and of all Betti numbers of $\AVOR 4$ with 
the exception of the middle Betti number $b_{10}$. This reduces the problem of the computation of the cohomology of $\AVOR 4$ to the computation of the Euler number of $\ab 4$, which is an independent problem in its
own right. Indeed, one can compute 
the Euler number of level covers ${\mathcal A}_4(n)$ for $n\geq 3$ by Hirzebruch--Mumford proportionality.
{}From this one could compute $e({\mathcal A}_4)$ if one had a complete classification of torsion elements 
in the group $\Sp(8,\Z)$. Although this is not known, it does not seem an impossible task to obtain such a 
classification. This is, however, a hard problem  which requires different methods from the 
the ones used in this paper; therefore, we will not approach it here.

Our approach involves computing a spectral sequence converging to the cohomology of $\AVOR 4$. As a by-product, this spectral sequence  determines the long exact sequence relating the cohomology with compact support of $\AVOR 4$ with that of $\ab4$ and of its boundary. 
For this reason, a natural application of our analysis is to obtain evidence on  the existence of non-trivial cohomology classes of $\ab4$ in different degrees, by giving \emph{lower bounds} for the Betti numbers of $\ab4$.
However, all contributions from the cohomology of $\ab4$ to our computation can be explained using either the Hodge classes or a certain non-algebraic class related to the one in $\coh{\ab3}6$ described in \cite[Thm 1]{Hain}.
This gives rise to the following question:
\begin{qs}
Is the rational cohomology of $\ab 4$ generated by the Hodge classes $\lambda_1,\lambda_3$ and by one non-algebraic cohomology class of Hodge weight $18$ arising in degree $12$? 
\end{qs}

The starting point of our investigations is the fact that every toroidal compactification $\ATOR g$
admits a  map $\phi_g\co \ATOR g \longrightarrow \ASAT g$. We recall that
$$
\ASAT g = \ab g \sqcup \ab {g-1} \sqcup \ldots \sqcup \ab 0,
$$
which allows us to construct a stratification of $\ATOR g$
by considering the closed loci $\beta^{(g)}_i=\beta_i=\phi_g^{-1}(\ASAT {g-i})$ and their open parts
$\beta_i^0=\beta_i \setminus \beta_{i+1}=\phi_g^{-1}(\ab {g-i})$. Each %
stratum 
$\beta_i^0$ is itself the disjoint union of locally closed substrata that are  quotients of torus bundles over the product of a certain number of copies of the universal family
$\cX_{g-i}$ over $\ab {g-i}$ by finite groups.  
The strategy is then to compute the cohomology with compact support 
of each of these substrata using Leray spectral sequences and then
to glue these strata by Gysin spectral sequences to compute the cohomology with compact support of $\beta_i^0$. 
Although this is a natural approach, we are not aware that it has been used in this form before apart 
from \cite{HT} where
it was applied to the case of genus $3$. The reader will however notice that the complexity encountered
in the present case is of a very different level: we need a full understanding of the Voronoi
decomposition in genus $4$, which in this case can no longer be deduced from the knowledge of the basic cone 
alone.

The use of Leray spectral sequences requires to know the cohomology with compact support of $\ab{g-i}$ not only with constant coefficients, but also with coefficients in certain symplectic local systems of low weight. In the case of $i=1,2$ we deduce this information from results on the cohomology of moduli spaces of pointed curves. Passing from the moduli space of curves to the moduli space of abelian varieties produces a small ambiguity, which does not influence our final result, mainly because it disappears at the level of Euler characteristics.
Up to this ambiguity, we are able to obtain complete results for the cohomology with compact support of all strata contained in the boundary as well as of the closure $\overline{\mathcal J}_4$ of the Jacobian locus in $\ab4$ and we 
believe that this is of some independent interest.

In the case of the cohomology of $\ab4$ itself, there are two facts which help us. The first is that the complement in $\ab4$ of the closure  of the locus of jacobians has  a smooth affine variety as coarse moduli space. This implies that its cohomology with compact support is trivial if the degree is smaller than $10$ and thus that the cohomology with compact support of $\ab4$ agrees with that of $\overline{\mathcal J}_4$ in degree $\leq 9$.
The second is that $\AVOR 4$ is %
(globally) the quotient of a smooth projective scheme by a finite group.
This implies that its cohomology satisfies Poincar\'e duality, and, more specifically, that 
its cohomology in degree $k$ carries a pure Hodge structure of weight $k$.

In the case of $\AVOR 4$, putting the cohomological information from all strata $\beta^0_i$ together  
yields Table~\ref{Biggysin}, from which we can deduce Theorem~\ref{mainthm} by using the Gysin
spectral sequence associated to the stratification given by the $\beta_i$. 
A consequence of this spectral sequence is that the cohomology groups with compact support of $\ab 4$ in degree $\leq 9$ are sufficient to determine all cohomology groups of $\AVOR 4$ of degree $\neq 10$. In turn, the cohomology groups of $\AVOR 4$ in degree $\geq 11$ so obtained can be used to gain information on the cohomology with compact support of $\ab4$ in degree $\geq 11$, thus leading (using Poincar\'e duality) to the question formulated above.
Finally, we 
obtain the Betti numbers for $\APERF 4$ in Theorem~\ref{mainthm_igu} by using the fact that $\AVOR 4$ is a
blow-up of $\APERF 4$ in one point.

The plan of the paper is as follows. 
In \S\ref{s:mainthm} we prove the main results using cohomological informations on the strata $\beta_i^0$. The geometrical study of each of the five strata $\beta_i^0$ is performed in the following five sections. Finally, in the appendix we collect and prove all results on the cohomology of local systems on $\ab2$ and $\ab3$ used in Section \ref{s:beta_1} and \ref{s:beta_2}.

\subsection*{Acknowledgments}
Partial support from DFG under grants Hu 337/6-1 and Hu 337/6-2 is gratefully acknowledged. During the preparation of this paper, the second author has been partially supported by the programme \emph{Wege in die Forschung II} of Leibniz Universit\"at Hannover. We would like to thank MSRI for hospitality during the first part of the preparation of this paper. 
Finally, we thank Frank Vallentin for untiringly answering our questions about the second Voronoi decomposition.

\subsection*{Notation}
\ 
\begin{longtable}{lp{10cm}}
$\ab g$ & moduli stack of principally polarized abelian varieties of genus~$g$\\[1pt]
$\cX_g$ & universal family over $\ab g$\\[1pt]
$\V_{\lambda_1,\dots,\lambda_g}$ & rational local system on $\ab g$ induced by the $\Sp(2g,\Q)$-representation indexed by the partition $(\lambda_1,\dots,\lambda_g)$\\[1pt]
$\ASAT g$ & Satake compactification of $\ab g$\\[1pt]
$\AVOR g$ &  Voronoi compactification of $\ab g$ \\[1pt]
$\XVOR g$& universal family over $\AVOR g$\\[1pt]
$\APERF g$ &  perfect cone compactification of $\ab g$ \\[1pt]
$\AIGU g$ &  Igusa compactification of $\ab g$ \\[1pt]
$\Mm gn$ & moduli stack of non-singular curves of genus $g$ with $n$ marked points\\[1pt]
\multicolumn{2}{l}{$\M g:=\Mm g0$}\\[1pt]
$\s_d$ & symmetric group in $d$ letters\\[1pt]
$\Sym_{\geq0}^2(\R^g)$ & space of real positive semidefinite quadratic forms in $\R^g$\\[1pt]
$\la \phi_1,\dots,\phi_r\ra$& convex cone generated by the half rays $\R_{\geq 0}\phi_1$, \dots, $\R_{\geq 0}\phi_r$
\end{longtable}
\smallskip
\setcounter{table}{0}

For every $g$, we denote by $\phi_g\co \AVOR g \rightarrow \ASAT g$ (respectively, $\psi_g\co \APERF g\rightarrow\ASAT g$) the natural map from the 
Voronoi (respectively, perfect cone) to the Satake compactification.
Let $\pi_g\co \XVOR g \to \AVOR g$ be the universal family, $q_g\co \XVOR g \to \XVOR g/\pm 1$ the quotient
map from the universal family to the universal Kummer family and $k_g\co \XVOR g/\pm 1 \to \AVOR g$
the universal Kummer morphism.

For $0\leq i\leq g$, we set $\beta_i^0=\phi_g^{-1}(\ab {g-i})\subset \AVOR g$, $\beta_i=\phi_g^{-1}(\ASAT {g-i})\subset \AVOR g$ and
$\beta_i^\Perf=\psi_g^{-1}(\ASAT {g-i})\subset \APERF g$.

We denote the Torelli map in genus $g$ by $\tau_g\co\M g\to \ab g$ its image, the 
Jacobian locus, by $\mathcal J_g = \tau_g(\M g)$ and closure of the image in $\ab g$ by $\overline {\mathcal J_g}$.

Throughout the paper, we work over the field $\C$ of complex numbers. All cohomology groups we consider will have rational coefficients. Since the rational cohomology of a Deligne--Mumford stack coincides with the rational cohomology of its coarse moduli space, we will sometimes abuse notation and denote stack and coarse moduli space with the same symbol.

In this paper, we make extensive use of mixed Hodge structures, focussing mainly on their weight filtration. We will denote by $\Q(-k)$ the Hodge structure of Tate of weight $2k$. For two mixed Hodge structures $A,B$ we will denote by $A\oplus B$ their direct sum and by $A+B$ any extension 
$$
0\rightarrow B \rightarrow E \rightarrow A\rightarrow 0.
$$
Furthermore, we will denote Tate twists of mixed Hodge structures by $A(-k) = A\otimes \Q(-k)$.

\section{Main theorems}\label{s:mainthm}

\begin{thm}\label{mainthm}
The cohomology of $\AVOR 4$ vanishes in odd degree and is algebraic in all even degrees, with the only possible exception of degree $10$. The Betti numbers are given by
$$
\begin{array}{c|ccccccccccc}
i&0&2&4&6&8&10&12&14&16&18&20\\\hline
b_i&1&3&5&11&17&10+e(\ab4)&17&11&5&3&1
\end{array}
$$
where $e(\ab4)$ denotes the Euler number of $\ab4$.
\end{thm}

The only missing information needed to compute all Betti numbers of $\AVOR 4$ is the Euler number.
As we shall see, we are able to compute the Euler numbers of all strata $\beta_i^0$ for $i\geq 1$, and thus, as already mentioned,
it would suffice to compute the Euler number of the space ${\mathcal A}_4$ itself (see the introduction for comments on this). 

\begin{thm}\label{mainthm_igu}
The Betti numbers of $\APERF 4$ in degree $\leq 9$ are given by
$$
\begin{array}{c|cccccccccc}
i&0&1&2&3&4&5&6&7&8&9\\\hline
b_i&1&0&2&0&3&0&8&0&14&0
\end{array}
$$
Moreover, all cohomology classes of degree $\leq 9$ are algebraic.
\end{thm}
\begin{cor}\label{ab4}
The rational cohomology of $\ab4$ vanishes in degree $1,11$ and all degrees $\geq13$. 
Furthermore one has $\coh{\ab 4}{0}=\Q(-10)$, $\coh{\ab 4}{2}=\Q(-9)$
and $\coh{\ab 4}{12}=\Q(-6)+\Q(-9)$. For the remaining Betti numbers a lower bound is given in 
the table 
$$
\begin{array}{c|cccccccccccccc}
i  &0&1&2&3&4&5&6&7&8&9&10&11&12&\geq 13\\\hline
b_i&1&0&1&\geq0&\geq1&\geq0&\geq2&\geq0&\geq1&\geq0&\geq1&0&2&0\\
\end{array}
$$
\end{cor}
Note that the vanishing of $b_1$ also follows independently from simple connectedness of $\ab 4$ 
\cite[Theorem 4.1]{HR} and that
$b_2=1$ corresponds to the fact that the Picard group is generated by the Hodge line bundle (over $\Q$).

\begin{proof}[Proof of Theorem~\ref{mainthm}]
To compute the cohomology of $\AVOR 4$, we study the Gysin spectral sequence 
$E^{p,q}_r\Rightarrow \coh{\AVOR4}{p+q}$ associated with the filtration
$\{T_i\}_{i=1,\dots,6}$ such that\begin{itemize}
\item $T_i=\beta_{5-i}$, $i=1,\dots,4$;
\item $T_5= \overline {\mathcal J_4} \cup T_4$;
\item $T_6=\AVOR 4$.
\end{itemize}

The $E_1$ term of this spectral sequence has the form $E^{p,q}_1 = \cohc{T_p\setminus T_{p-1}}{p+q}$. 
For $p=1,\dots,4$ the strata $T_p\setminus T_{p-1}$ coincide with the strata of $\AVOR 4$ of semi-abelic varieties of torus rank $5-p$; their cohomology with compact support is computed in the next sections by combining combinatorial information on the toroidal compactification with the geometry of fibrations on moduli spaces of abelian varieties (see Propositions~\ref{cohom_beta1}, \ref{cohom_beta2}, \ref{cohom_beta3} and Theorem \ref{Voronoi-beta4}).
The stratum $T_5\setminus T_4$ is the closure inside $\ab4$ of the locus of jacobians. Its cohomology with compact support is computed in Lemma~\ref{jacobians}.

The only remaining stratum is the open  stratum $T_6\setminus T_5$. Let  ${\mathcal J}^{\Sat}_4$ be the closure of  $\mathcal J_4$
in $\ASAT 4$. Since this contains the entire boundary of $\ASAT 4$ it follows that 
$$
T_6 \setminus T_5=\ab 4 \setminus \overline {\mathcal J}_4= \ASAT 4 \setminus {\mathcal J}^{\Sat}_4.
$$
The latter set is affine since it is the complement of an ample hypersurface on $\ASAT 4$ 
 (see \cite{HaHu}).
In particular, its cohomology with compact support can be non-trivial only if the degree lies between $10$ and $20$.

From this it follows that the $E_1$ term of the Gysin spectral sequence associated with the filtration $\{T_i\}$ is as given in Table~\ref{Biggysin}. For the sake of simplicity, in that table we have denoted $\cohc[\V_{2,2}]{\ab2}3$ and $\cohc[\V_{2,2}]{\ab2}4$ with the same symbol $H$, even though a priori they are only isomorphic after passing to the associated graded piece with respect to the weight filtration. (Furthermore, the results in in \cite{OTM24} imply $H=0$.)

\begin{table}
\caption{\label{Biggysin}$E_1$ term of the Gysin spectral sequence associated with the filtration $T_i$.}
\rotatebox{90}
{
\begin{tabular}{l}$
\begin{array}{r|ccccccc}
q&\\[6pt]
17&\Q(-9) & 0     &   0   & 0    &  0   &  0  &\\
16&  0    & 0     &   0   & 0    &  0   &  0  &\\
15&\Q(-8)& 0     &   0   &    0 &  0   &  0  &\\
14&  0    & 0     &   0   &    \Q(-9) &  0   &\Q(-10)&\\
13&\Q(-7)^{\oplus2}& 0     &\Q(-8) &   0  &\Q(-9)    &  ?  &\\
12&  0    &\Q(-7) &   0   &   \Q(-8)^{\oplus2}&0&  ?  &\\
11&\Q(-6)^{\oplus4}&  0    &\Q(-7)^{\oplus3}& 0 &\Q(-8)    &  ?  &\\
10&  0    &\Q(-6)^{\oplus3}&   0   &   \Q(-7)^{\oplus3}  &0&  ?  &\\
 9&\Q(-5)^{\oplus6}&  0    &\Q(-6)^{\oplus5}&\Q(-6)^{\oplus\epsilon} &  \Q(-7)^{\oplus2}   &  ?  &\\
 8&  0    &\Q(-5)^{\oplus4}&\Q(-3) &   
{
\Q(-6)^{\oplus(4+\epsilon)}+\Q(-3)
}   &0& ?  &\\
 7&\Q(-4)^{\oplus7}&0      &\Q(-5)^{\oplus5}+H(-1)
                          &\Q(-5)^{\oplus\epsilon}& \Q(-6)   &  ?  &\\
 6&  0    &\Q(-4)^{\oplus4}& \Q(-2)+H(-1)
                          &   {\Q(-5)^{\oplus(3+\epsilon)}+\Q(-2)} & 0 &  ?  &\\
 5&{\Q(-3)^{\oplus6}
   +\Q(-1)}& \Q(-1)&{\Q(-4)^{\oplus3}
                   +\Q(-2)+H
                          }&  \Q(-2)    &  \Q(-5)   &   ? &\\
 4&  0    &\Q(-3)^{\oplus2}& \Q(-1)+H& { \Q(-4)^{\oplus2}+\Q(-1)}  &0 &  ? &\\
 3&\Q(-2)^{\oplus3}&  \Q   &\Q(-3)^{\oplus2}&  \Q(-1)   &  {\Q(-4)+\Q(-1)}   &  0  &\\
 2&  0    & \Q(-2)&   0   &    {\Q(-3)+\Q}  & 0 &  0  &\\
 1&\Q(-1)^{\oplus2}&   0   &\Q(-2) &   0  &   0  &  0  &\\
 0&  0    &\Q(-1) &   0   &   0  &   0  &  0  &\\
-1& \Q    &  0    &   0   &   0  &   0  &  0  &\\
\hline
  &  1  &  2  &  3  &  4  &  5  &  6  &p\\
\end{array}
$\\
$H=\cohc[\V_{2,2}]{\ab2}3\cong\cohc[\V_{2,2}]{\ab2}4$ (up to grading), $\epsilon:=\rank\cohc[\V_{1,1,0}]{\ab3}9$.
\end{tabular}
}
\end{table}
Since the terms in the sixth column are only known for $q\leq3$, in the following we will only deal with the terms of the spectral sequence that are independent of them, that is, the $E^{p,q}_r$ terms with $p+q\leq 8$. 

Let us recall that $\AVOR 4$ is a smooth Deligne--Mumford stack which is globally the quotient of a smooth proper variety by a finite group. From this it follows that the cohomology groups of $\AVOR 4$ carry pure Hodge structures of weight equal to the degree. Therefore, the Hodge structures on $E^{p,q}_\infty$ have to be pure of weight $p+q$. This means that for all $p,q$, the graded pieces of $E^{p,q}_1$ of weight different from $p+q$ are killed by differentials. In particular, if we restrict to the range $p+q\leq 9$, this gives that the $E_\infty$ terms are as given in Table~\ref{einfty}. Of course, this does not describe precisely at which $E_r$ the spectral sequence degenerates, or what exactly is the rank of the differentials. For instance, if one assumes $H=0$ (which in view of the results in \cite{OTM24} is indeed the case), a natural thing to expect is that the $d_1$-differentials $E^{1,5}_1\rightarrow E^{2,5}_1$, $E^{3,5}_1\rightarrow E^{4,5}_1$, $E^{3,4}_1\rightarrow E^{4,4}_1$ and $E^{4,3}_1\rightarrow E^{5,3}_1$, as well as the $d_2$-differential $E^{2,3}_2\rightarrow E^{4,2}_2$ have rank $1$, but this is not the only possibility. 
The claim on the cohomology of $\AVOR 4 $ in degree $\leq 9$ follows from the $E_\infty$ term in Table~\ref{einfty}. The claim on the cohomology in degree $\geq 11$ follows by Poincar\'e duality.
Finally, the computation relating the middle Betti number $b_{10}$ to the Euler number $e(\ab4)$ follows from the additivity of the Betti numbers of the stratification $\{T_\pu\}$.
\begin{table}
\caption{$E^{p,q}_\infty$ in the range $p+q\leq 9$. \label{einfty}}
$
\begin{array}{r|ccccccc}
q&\\[6pt]
 8&  0    \\
 7&\Q(-4)^{\oplus7}&0      \\
 6&  0    &\Q(-4)^{\oplus4}& 0\\
 5&\Q(-3)^{\oplus6}&   0   &\Q(-4)^{\oplus3}& 0 \\
 4&  0    &\Q(-3)^{\oplus2}&  0    &\Q(-4)^{\oplus2}&0 &   &\\
 3&\Q(-2)^{\oplus3}&   0   &\Q(-3)^{\oplus2}& 0     &\Q(-4) &  0  &\\
 2&  0    & \Q(-2)&   0   & \Q(-3)& 0 &  0  &\\
 1&\Q(-1)^{\oplus2}&   0   &\Q(-2) &   0  &   0  &  0  &\\
 0&  0    &\Q(-1) &   0   &   0  &   0  &  0  &\\
-1& \Q    &  0    &   0   &   0  &   0  &  0  &\\
\hline
  &  1  &  2  &  3  &  4  &  5  &  6  &p\\
\end{array}$
\end{table}
\end{proof}
\begin{rem}
The fact that the mixed Hodge structures on the $E_1^{p,q}$ in Table~\ref{Biggysin} are compatible with obtaining $E_\infty^{p,q}$ terms that carry pure Hodge stuctures of the correct weight provide an important check on the correctness of our computations. 
\end{rem}

\begin{proof}[Proof of Theorem~\ref{mainthm_igu}]
The proof is analogous to that of Theorem~\ref{mainthm}. Rather than working with the filtration $\{T_i\}$, we will consider the stratification $\{T_i^\Perf\}$ 
defined analogously by $T_i^\Perf=\beta_{5-i}^\Perf$ for $1\leq i\leq 4$ and $T_5^\Perf=\overline{\mathcal J}_4\cup T_4^\Perf$, $T_6^\Perf=\APERF 4$. 
The closed stratum $T_1$ is the the locus $\beta_4^\Perf$ of torus rank $4$ inside $\APERF 4$. Hence $E^{1,q}_1 = \coh{\beta_4^\Perf}{q+1}$ can be obtained from Theorem~\ref{Igusa-beta4}. 

Since the exceptional divisor of the blow-up map $q\co\AVOR 4\rightarrow\APERF 4$ is contained in $T_1$, we have $(\APERF 4\setminus q(T_1))\cong (\AVOR 4\setminus T_1)$. In particular, the Gysin spectral sequence associated with the stratification of $\APERF 4$ has $E^{p,q}_1$ terms that coincide with those of Table~\ref{Biggysin} for $p\geq 2$. Moreover, also the rank of all differentials $E^{p,q}_r\rightarrow E^{p+r,q-r+1}_r$ coincide with those for the filtration $\{T_i\}$ as long as no $E^{1,q}_r$-terms are involved. This already implies the claim for all degrees different from $6$. In degree $6$, it is necessary to decide whether the class of Hodge weight $2$ in $E^{5,1}_1$ is killed by differentials of the spectral sequence or not. If we consider the map $\AVOR 4\supset \beta_4\rightarrow\beta_4^\Perf\subset \APERF 4$, we have that the weight $2$ class on $\beta_4^\Perf$ lies in the image of the weight $2$ class in the cohomology of $\beta_4$, which was killed by differentials for purity reasons on $\AVOR 4$. This implies that this must be the case also on $\APERF 4$. From this the claim follows.
\end{proof}

\begin{proof}[Proof of Corollary~\ref{ab4}]
The lower bounds in the claim are those given by the dimension of the subring generated by $\lambda_1$ and $\lambda_3$. 
The vanishing of the cohomology of $\ab4$ in all degrees $i\geq 13$ and in degree $11$, as well as $\coh{\ab4}{12}=\Q(-6)+\Q(-9)$ follow directly from the last two columns of Table~\ref{Biggysin} by Poincar\'e duality.  
\end{proof}

\begin{rem}
Comparing Table~\ref{Biggysin} with the cohomology of $\AVOR 4$ of degree $\geq 11$ suggests that the cohomology of the open stratum $\ab4\setminus\overline{\mathcal J}_4$ could vanish in all positive degrees, with the exception of degree $10$ on which Poincar\'e duality yields no information.

The reason is the following.
One can reinterpret the first 4 columns of Table~\ref{Biggysin} as the $E_1$ terms of a Gysin spectral sequence converging to the cohomology of the boundary $\partial\AVOR4=\AVOR4\setminus\ab4$. Then the remaining information from that Table is equivalent to the study of the long exact sequence 
\begin{equation}\label{inclbdry}
\cohc{\ab4}k\rightarrow\coh{\AVOR 4}k\rightarrow\coh{\partial\AVOR4}k\rightarrow\cohc{\ab4}{k+1}
\end{equation}
associated with the closed inclusion of the boundary in $\AVOR 4$. 
One can compare the information on the cohomology of the boundary coming from Table~\ref{Biggysin} with the Betti numbers of $\AVOR4$ in degree $\geq 11$ from Theorem~\ref{mainthm}. Then one sees that even in this range the results are compatible with the vanishing of the cohomology of $\ab4\setminus\overline{\mathcal J}_4$, or, equivalently, with the hypothesis that the cohomology of $\ab4$ is the minimal possible, i.e. generated by $\lambda_1$, $\lambda_3$ and $\alpha\in H^{9,9}(\coh{\ab4}{12})$ from Corollary~\ref{ab4}.
This would imply that the cohomology of $\ab4$ coincides with the stable cohomology in degree $\leq 10$, while a priori this is known only in degree $\leq g-2=2$. 

Furthermore, \eqref{inclbdry} gives strong restrictions on the possible existence of cohomology classes on $\ab4$ that are not in the subring generated by the Hodge classes. This follows from the purity of the cohomology of $\AVOR 4$ combined with Table~\ref{Biggysin}, which ensures that the cohomology of $\partial\AVOR 4$ is very close to be pure itself. Practically, this forces non-trivial cohomology classes from $\ab4\setminus\overline{\mathcal J}_4$ to appear in pairs of isomorphic Hodge structures, occurring as  graded pieces of $\coh {\ab4\setminus\overline{\mathcal J}_4}k$ and $\coh {\ab4\setminus\overline{\mathcal J}_4}{k+1}$.

One could also use
Table~\ref{Biggysin} to prove the vanishing of $\coh{\ab4\setminus\overline{\mathcal J}_4}k$. Then one needs to prove that all (algebraic) classes of weight $10-k$ that occur in $E_1^{p,q}$ with $p+q=10-k$ give rise to cohomology classes in $\cohc{\AVOR 4}{10-k}$ that are linearly independent. This is known for divisors (ensuring the vanishing of $\coh{\ab4\setminus\overline{\mathcal J}_4}k$ for $k=1,2$). It would be interesting to investigate it for classes of higher codimension.

Note that, if one knew that $\coh{\ab4\setminus\overline{\mathcal J}_4}k$ vanishes for all $1\leq k\leq 9$, then this would yield the following result for the Betti numbers of $\APERF 4$ in degree $\geq 11$:
$$
\begin{array}{c|ccccc}
i  &12&14&16&18&20\\\hline
b_i&14&9&4&2&1
\end{array}
$$
as well as the vanishing of all odd Betti numbers of $\APERF 4$.
\end{rem}

\section{Torus rank $0$}

We start by considering $T_5\setminus T_4$, which is the Zariski closure $\overline {\mathcal J}_4$ of the locus of jacobians $\mathcal J_4 = \tau_4(\M4)$ inside $\ab4$.

\begin{lem}\label{jacobians}
The only non-zero Betti numbers with compact support of $\overline {\mathcal J}_4$ are as follows:
$$\begin{array}{c|cccccc}
i&18&16&14&12&10&8\\\hline
b_i&1&1&2&1&1&2
\end{array}
$$
In particular, all odd Betti numbers vanish. 

Furthermore, all cohomology groups with compact support are generated by algebraic classes, with the only exception of $\cohc{\overline{\mathcal J}_4}8$, which is an extension of $\Q(-4)$ by $\Q(-1)$.
\end{lem}
\proof
We compute the cohomology with compact support of $\overline{\mathcal J}_4$ by recalling that the Zariski closure of 
the locus of jacobians in $\ab4$ is the union of the image of the Torelli map and the locus of abelian fourfolds that 
are products of abelian varieties of dimension $\leq 3$. 
This allows to cover $\overline{\mathcal J_4}$ by the following locally closed disjoint strata:
$$
\begin{array}{l}S_1=\Sym^4\ab1,\ S_2=\tau_2(\M2)\times\Sym^2\ab1,\ S_3=\Sym^2\tau(\M2),\\
 S_4=\tau_3(\M3)\times\ab1,\ S_5=\tau_4(\M4).
\end{array}$$

Furthermore, the Torelli map in all genera induces an isomorphism in cohomology with rational coefficients 
between $\M g $ and its image $\tau_g(\M g)$. This allows to compute the cohomology with compact support of all 
strata from previously known results on the cohomology of $\M g$ with 
$g\leq 4$ (\cite{Mu-curves},\cite{Looij},\cite{OTM4}). 
These yield that the $E_1$ term $E^{p,q}_1 = \cohc{S_p}{p+q}$ of the Gysin exact sequence of the filtration 
associated with the stratification $S_j$  is as in Table~\ref{Gysin-T5}.

\begin{table}
\caption{\label{Gysin-T5}$E_1$ term of the Gysin spectral sequence converging to the cohomology with compact support of the closure of the locus of jacobians in $\ab4$}
$$\begin{array}{r|cccccc}
q\\[6pt]
13&0&0&0&0&\Q(-9)&\\
12&0&0&0&0&0&\\
11&0&0&0&0&\Q(-8)&\\
10&0&0&0&\Q(-7)&0&\\
 9&0&0&\Q(-6)&0&\Q(-7)&\\
 8&0&\Q(-5)&0&\Q(-6)&\Q(-6)&\\
 7&\Q(-4)&0&0&0&0&\\
 6&0&0&0&0&0&\\
 5&0&0&0&0&0&\\
 4&0&0&0&\Q(-1)&0&\\
\hline
&1&2&3&4&5&p
\end{array}
$$
\end{table}

In view of Table~\ref{Gysin-T5}, to calculate the cohomology with compact support of
$\overline{\mathcal J_4}$ it is sufficient to know the rank of the differential
$$d:\;\cohc{\overline{\mathcal
J}_4^\op{red}}{12}\cong\Q(-6)^{\oplus2}\longrightarrow\cohc{\mathcal J_4}{13}\cong\Q(-6)$$
in the Gysin long exact sequence associated with the closed inclusion of the locus
$\overline{\mathcal J}_4^\op{red}=
\overline{\mathcal J}_4 \setminus {\mathcal J}_4 = \overline S_3\cup \overline S_4\subset\ab4$
of reducible abelian fourfolds in the Zariski closure in $\ab4$ of the locus of jacobians
$\mathcal J_4 = \tau_4(\M 4)$.

We observe that $\cohc{\overline{\mathcal J}_4^\op{red}}{12}$ is generated by two
$6$-dimensional algebraic cycles $C_1$ and $C_2$, where $C_1$ is the fundamental class of
$\overline{S_3}$ and  $C_2$ the fundamental class of 
$\overline{\tau(\mathcal H_3)}\times\ab1$, where $\mathcal H_3$ is the hyperelliptic locus. 
Therefore, the surjectivity of $d$ is equivalent to the existence of a
relation between $C_1$ and $C_2$ viewed as elements of the Chow group of
$\overline{\mathcal J}_4$.

Let us denote by $\M4^\op{ct}$ the moduli space of stable genus $4$ curves of compact type, i.e. such that that their generalized Jacobian is compact.  
Then the Torelli map extends to a proper morphism
$$\tau^\op{ct}\co\M4^\op{ct}\longrightarrow\overline{\mathcal J}_4.$$
From the geometric description of the map $\tau^\op{ct}$ it follows that the image under
$\tau^\op{ct}$ of the Chow group of dimension $6$ cycles supported on the boundary
$\M4^\op{ct}\setminus\M4$ coincides with $\langle C_1,C_2\rangle$. 
Indeed, let $D_1$ be the closure of the locus of stable curves consisting of two genus $2$ curves
intersecting in a Weierstrass point and let $D_2$ be the closure of the locus of stable curves consisting
of elliptic curves intersecting a hyperelliptic genus $3$ curve in a Weierstrass point. Then $D_1$ and $D_2$ map 
to $C_1$ and $C_2$ respectively. It is known that the dimension $6$ classes in $\M4^\op{ct}$
fulfill a relation, given by the restriction of the relation on $\Mb 4$ of \cite[Prop.~2]{yang}. 
When pushed forward via $\tau^\op{ct}$, this relation gives a non-trivial relation between $C_1$ and $C_2$.
Thus the differential $d$ has to be surjective and the claim follows. 
\qed

\section{Torus rank $1$}\label{s:beta_1}

Next, we deal with the locus $\beta^1_0$ of semi-abelic varieties of torus rank $1$.
\begin{prop}\label{cohom_beta1}
The rational cohomology with compact support of $\beta_1^0$ is as follows: the non-zero Betti numbers are
$$
\begin{array}{c|ccccccccccc}
i&6&7&8&9&10&11&12&13&14&16&18\\\hline
b_i&2&1&3&1&4+\epsilon&\epsilon&5+\epsilon&\epsilon&3&2&1
\end{array}
$$
where $\epsilon=\rank \cohc[\V_{1,1,0}]{\ab3}9$.
The cohomology groups of even degree $2k$ are algebraic for $k\geq 7$; for $k\leq 6$ they are extensions of pure Hodge structures of the form $\cohc{\beta_2^0}{2k}=\Q(-k)^{\oplus(b_{2k}-1)}+\Q(k-3)$. The Hodge structures in odd degree are given by $\cohc{\beta_2^0}{2k+1}=\Q(2-k)$ for $k=7,9$ and 
$\cohc{\beta_2^0}{2k+1}=\Q(-k)^{\oplus\epsilon}$ for  $k=11,13$.
\end{prop}

\proof
To compute the cohomology with compact support of $\beta_1^0$ we will use the map $k_3\co \beta_1^0\rightarrow \ab 3$ realizing $\beta_1^0$ as the universal Kummer variety over $\ab 3$. The fibre of $\beta_1^0$ over a point parametrizing an abelian surface $S$ is $K:=S/\pm 1$. 

Note that the cohomology of $K$ vanishes in odd degree because of the Kummer involution.
The cohomology of $K$ is one-dimensional in degree $0$ and $6$ and induces trivial local systems on $\ab3$. 
The cohomology group $H^2(K;\Q)\cong \bigwedge^2H^1(S;\Q)$ is $15$-dimensional and induces the local system $\V_{1,1,0}\oplus \Q(-1)$ on $\ab 3$. 
By Poincar\'e duality we have $\coh K4\cong\coh K2\otimes\Q(-1)$, inducing the local system $\V_{1,1,0}(-1)\oplus\Q(-2)$ on $\ab3$.

The cohomology with compact support of $\ab3$ in the local system $\V_{1,1,0}$ is calculated in Lemma~\ref{ab3-v110}.
We refer to Theorem~\ref{hain} for the cohomology with compact support of $\ab3$ with constant coefficients, which was calculated by Hain in \cite{Hain}.
These results allow to compute $E^{p,q}_2=\cohc[R^q_!{k_3}_*(\Q)]{\ab3}{p}$ for the Leray spectral sequence $E_\pu^{p,q}\Rightarrow \cohc{\beta_1^0}{p+q}$ associated with $k_3\co \beta_1^0\rightarrow \ab 3$.  This $E_2$ term is given in Table~\ref{E2-beta1}.
\begin{table}
\caption{\label{E2-beta1}$E_2$ term of the Leray spectral sequence converging to $\cohc{\beta_1^0}\pu$. }
\scalebox{0.96}{
$
\begin{array}{r|ccccccccc}
q&\\[6pt]
6&0&\Q(-6)+\Q(-3)&0&\Q(-7)&0&\Q(-8)&0&\Q(-9)\\
5&0&0&0&0&0&0&0&0&\\
4&\Q(-2)&\Q(-5)+\Q(-2)&0&\Q(-6)^{\oplus (1+\epsilon)}&\Q(-6)^{\oplus\epsilon}&\Q(-7)&0&\Q(-8)\\
3&0&0&0&0&0&0&0&0&\\
2&\Q(-1)&\Q(-4)+\Q(-1)&0&\Q(-5)^{\oplus (1+\epsilon)}&\Q(-5)^{\oplus\epsilon}&\Q(-6)&0&\Q(-7)\\
1&0&0&0&0&0&0&0&0&\\
0&0&\Q(-3)+\Q&0&\Q(-4)&0&\Q(-5)&0&\Q(-6)\\
\hline
&5&6&7&8&9&10&11&12&p
\end{array}
$}
$\epsilon = \rank\cohc[\V_{1,1,0}]{\ab3}9\in\{0,1\}$.
\end{table}
Note that all differentials of this Leray spectral sequence vanish for Hodge-theoretic reasons, so that $E_2=E_\infty$. Specifically, all differentials must involve one $E_2^{p,q}$ term with $p+q$ odd, but there are only two such terms, namely $E_2^{5,2}$ and $E_2^{5,4}$. 
It follows from Table~\ref{E2-beta1} that for both these terms, any differential $d_k$ with $k\geq 2$ which involves one of them will map either to $0$ or to a $E_2$ term that carries a pure Hodge structure of different weight. In both cases the differential has to be $0$.  
\qed

\section{Torus rank $2$}\label{s:beta_2}

In this section we compute the cohomology with compact support of the stratum $\beta^0_2$ of $\AVOR 4$ of rank $2$ degenerations of abelian fourfolds. For this purpose we recall first the known global construction of $\beta^2_0$ as the quotient of a $\Pp1$-bundle of a fibre product of the universal family over $\ab2$.

Furthermore, let us recall that the restriction of the Voronoi fan in genus $g$ to $\Sym^2_{\geq0}(\R^{g'})$ for genus $g\geq g'$ coincides with the Voronoi fan in genus $g'$. This implies that the geometric constructions of the fibrations $\beta^0_2\rightarrow \ab2$ and $\beta^0_3\rightarrow\ab1$ we give in this section and in the following one, respectively, are actually independent of the choice of $g=4$ but extend to analogous descriptions of the fibres of fibrations $\beta^0_2\rightarrow \ab{g-2}$ and $\beta^0_3\rightarrow\ab{g-3}$ that exist for $\beta^0_2,\beta^0_3\subset\AVOR g$ independently of $g$. 
 In particular, the geometric construction of $\beta^0_2$ explained in this section coincides with the construction used in \cite[\S4]{HT} to compute the cohomology with compact support of the corresponding locus in $\AVOR 3$. 

\begin{prop}\label{cohom_beta2}
The rational cohomology with compact support of $\beta_2^0$ is as follows: the non-zero Betti numbers are
$$
\begin{array}{c|ccccccccccc}
i&4&6&7&8&9&10&11&12&14&16\\\hline
b_i&1&2&1+r&4+r&1+r&5+r&1&5&3&1
\end{array}
$$
where $r=\rank \cohc[\V_{2,2}]{\ab2}3$.
If we assume $r=0$, then all cohomology groups of even degree are algebraic, with the exception of $\cohc{\beta_2^0}8=\Q(-4)^{\oplus3}+\Q(-2)$ which is an extension of Hodge structures of Tate type. The Hodge structure in odd degree $2k+1$ with $k=3,4,5$ is pure of Tate type of weight $2k-4$.
\begin{rem}
It follows from \cite{OTM24} that $r=0$.
\end{rem}
\end{prop}
The proof of this Proposition will given in \S\ref{proof-beta2} after some preliminary  steps.

In the previous section, we calculated the cohomology with compact support of $\beta^0_1$ using the map $k_3\co \beta^0_1\rightarrow \ab3$ given by the universal Kummer variety. This map extends to the stratum $\beta_2^0$ of degenerations of abelian fourfolds of torus rank $2$, giving a map $k_3\co (\beta_1\setminus\beta_3)\rightarrow \AVOR3$.
Under this map, the elements of $\AVOR 4$ with torus rank $2$ are mapped to elements of $\AVOR 3$ of torus rank $1$. If we denote by ${\beta'_t}^0$ the stratum of $\AVOR 3$ of semi-abelian varieties of torus rank exactly $t$, we get a commutative diagram
$$\xymatrix@R=10pt@C=18pt{
{\AVOR 4}\ar@{<-^{)}}[d] &{\AVOR 3}\ar@{<-^{)}}[d] &{\AVOR 2}\ar@{<-^{)}}[d] \\
{\beta_2^0} \ar[r]^{k_3} & {{\beta'_1}^0}\ar[r]^{k_2}& {\ab 2}\\
{}&{}&{}\\
{\pi_3^{-1}({\beta'_1}^0)}\ar[uu]^{q_3}\ar[uur]_{\pi_3}&{{\mathcal X}_2}\ar[uu]^{q_2}\ar[uur]_{\pi_2}&{}
}$$

The map $\pi_3$ is the restriction of the universal family over $\AVOR 3$. In particular, the fibres of $\pi_3$ over points of ${\beta'_1}^0$ are rank $1$ degenerations of abelian threefolds, i.e. compactified $\C^*$-bundles over abelian surfaces. A geometric description of these compactified $\C^*$-bundles is given in \cite{Mu}. They are obtained by taking the $\Pp1$-bundle associated to the $\C^*$-bundle and then gluing the $0$- and the $\infty$-section with a shift, defined by a point of the underlying abelian surface that is uniquely determined by the line bundle associated to the $\C^*$-bundle.
In particular, this shift is $0$ for the fibres of the $\pi_3$ over the $0$-section of the Kummer fibration ${\beta'_1}^0\cong(\cX_2/\pm1)\xrightarrow{k_2} \ab2$, which are thus products of a nodal curve and an abelian surface. 

We want to describe the situation in more detail. For this, 
consider the universal Poincar\'e bundle $\mathcal P\rightarrow \cX_2\times_{\ab 2}\hat{\cX_2}$, normalized so that the restriction to the zero section $\hat{\cX_2}\rightarrow \cX_2\times_{\ab 2}\hat{\cX_2}$ is trivial. 
Let 
${\overline U}=\bb P(\mathcal P\oplus \mathcal{O}_{\cX_2\times_{\ab2}\hat{\cX_2}})$ be the associated ${\bb P}^1$-bundle. 
Using the principal polarization we can naturally 
identify $\hat{\cX_2}$ and ${\cX_2}$, which we will do from now on.
We denote by 
$\Delta$ the union of the $0$-section and the $\infty$-section of this bundle. 
Set $U= {\overline U} \setminus \Delta$, which is simply the $\C^*$-bundle 
given by the universal Poincar\'e bundle $\mathcal P$ with the $0$-section removed and denote the 
bundle map by $f\co U\rightarrow \cX_2\times_{\ab2}\cX_2$.
Then there is a map $\rho\co \overline U \to 
\beta_2^0$ with finite fibres. 
Note that the two components of $\Delta$ are identified under the map $\rho$. The restriction of $\rho$ to both $U$ and to
$\Delta$ is given by a finite group action, although the group is not the same in the two cases (see the discussion below). 

\subsection{Geometry of the $\C^*$-bundle}
We now consider the situation over a fixed point $[S] \in \ab2$.
For a fixed degree $0$ line bundle ${\mathcal L}_0$ on $S$ 
the preimage $f^{-1}(S \times \{{\mathcal L}_0 \})$ is a semi-abelian threefold, namely the $\C^*$-bundle given by 
the extension corresponding to  ${\mathcal L}_0 \in \hat{S}$.  This semi-abelian threefold admits a Kummer involution
$\iota$ which acts as $x \mapsto -x$ on the base $S$ and by $t \mapsto 1/t$ on the fibre over the zero section. The Kummer involution $\iota$ is defined universally on $U$.

Consider the two involutions $i_1, i_2$ on $\cX_2\times_{\ab2}\cX_2$ defined by 
$$i_1(S,p,q)=(S,-p,-q) \text{ and } i_2(S,p,q)=(S,q,p)$$ 
for every abelian surface $S$ and every $p,q\in S$. These two involutions lift to involutions $j_1$ and $j_2$ on $U$ that act trivially on the fibre of $f\co U\rightarrow \cX_2\times_{\ab2}\cX_2$ over the zero section. 

The following lemma can also be proved directly from the toroidal construction of $\AVOR 4$ using the approach of \cite[Lemma 2.4]{S-B}.

\begin{lem}
The diagram
\begin{equation}\label{poincare}
\xymatrix{
U \ar[r]_{f} \ar @/^1pc/ [rr]^{g} \ar[d]^{\rho|_U} & {\cX_2\times_{\ab2}\cX_2}\ar[d]^{\rho'}\ar[r] & {\ab 2}\\
{\beta_2^0 \setminus \rho(\Delta)} \ar[r] & {\Sym^2_{\ab2}(\cX_2/\pm 1)}\ar[ru],\\
}
\end{equation}
where $\rho'\co{\cX_2\times_{\ab2}\cX_2}\to{\Sym^2_{\ab2}(\cX_2/\pm 1)}$ is the natural map, is commutative. Moreover 
$\rho|_U\co U\rightarrow \rho(U)=\beta_2^0\setminus\rho(\Delta)$ is the quotient of $U$ by the subgroup of the automorphism group of $U$ generated by $\iota, j_1$ and $j_2$.
\end{lem}

\begin{proof}
Since the map $\rho'$ in the diagram~\eqref{poincare} has degree $8$ and $\iota, j_1,j_2$ generate a subgroup of order $8$ of the automorphism group of $U$, it suffices to show that the map $\rho|_U$ factors through each of the involutions $\iota$ and $j_1, j_2$. 
 
Recall that the elements of $\beta_2^0$ correspond to rank $2$ degenerations of abelian fourfolds. 
More precisely, every point of $\rho(U)$ corresponds to a degenerate abelian fourfold $X$ whose 
normalization is a ${\bb P}^1 \times {\bb P}^1$-bundle, namely the 
compactification of a product of two $\C^*$-bundles on the abelian surface $S$ given by $k_1\circ k_2([X])$. 
The degenerate abelian threefold itself is given by identifying 
the $0$-sections and the $\infty$-sections of the ${\bb P}^1 \times {\bb P}^1$-bundle. This identification 
is determined by a complex parameter, namely the point on a fibre of $f\co U\rightarrow \cX_2\times_{\ab2}\cX_2$.

Since a degree $0$ line bundle ${\mathcal L}_0$ and its inverse define isomorphic semi-abelian threefolds and since
the role of the two line bundles is symmetric, the map $\rho|_U$ factors through $\iota$ and $j_2$. 
Since $j_1$ is the commutator of $\iota$ and $j_2$ the map $\rho|_U$ also factors through $j_1$.
\end{proof}

We will compute the cohomology with compact support of $\beta^0_2$ by considering the Leray spectral sequence associated with the fibration $k_2\circ k_3\co \beta^0_2\rightarrow\ab2$. 
This requires to compute the cohomology with compact support of the fibre $(k_2\circ k_3)^{-1}([S])$ over a point $[S]\in\ab2$. To this end, we decompose $(k_2\circ k_3)^{-1}([S])$ into an open part given by its intersection with $\rho(U)$ and a closed part given by its complement.

\subsection{Cohomology of the open part of the fibre} %
The fibration $g\co U\rightarrow\ab2$ obtained by composing the $\C^*$-bundle $f\co U\rightarrow \cX_2\times_{\ab2}\cX_2$ with the natural map $\cX_2\times_{\ab2}\cX_2\rightarrow\ab2$ plays an important role in the study of the restriction of $k_2\circ k_3$ to $\rho(U)$.
Namely, the fibre of $(k_2\circ k_3)|_{\rho(U)}$ over $[S]\in\ab2$ coincides with the quotient of the fibre of $g$ under the automorphism group generated by $j_1,j_2$ and $\iota$. Therefore, the cohomology of the fibre of $k_2\circ k_3$ restricted to $\rho(U)$ is the part of the cohomology of $g^{-1}([S])$ that is invariant under $j_1,j_2$ and $\iota$.

We start by computing the actions of $i_1$, $i_2$ and of the involution $\kappa:\;(p,q)\mapsto(-p,q)$ induced by the Kummer involution of semi-abelian threefolds of torus rank $1$ on the cohomology of $S\times S$. Recall that the cohomology of $S$ is isomorphic to the exterior algebra generated by the $6$-dimensional space $\Lambda:=H^1(S;\Q)$ and that $H^\pu(S\times S;\Q)\cong H^\pu(S;\Q)^{\otimes2}$ by the K\"unneth formula. Using this description, one can calculate the part of the cohomology of $S\times S$ which is invariant under $i_1$ and $i_2$. In particular, since all cohomology in odd degree is alternating under the involution $i_1$, the only non-trivial invariant cohomology groups are in even degree. We give the description of the invariant cohomology groups in the second column of Table~\ref{SxS}. One then proceeds to investigate their structure with respect to $\kappa$. For instance one can use the isomorphism
$H^k(S\times S;\Q)^{(i_1,i_2,\kappa)}\cong H^k(S\times S/(i_1,i_2,\kappa);\Q)$,
together with the fact that the quotient of $S\times S$ by the subgroup generated by $i_1,i_2$ and $\kappa$ is the second symmetric product of $S/\pm1$. In this way one proves that the behaviour of the cohomology with respect to $\kappa$ is as given in the last two columns of Table~\ref{SxS}.

\begin{table}
$$\begin{array}{llll}
k & H^k(S\times S;\Q)^{(i_1,i_2)} & \kappa\text{-invariant} & \kappa\text{-alternating}\\
8 & \Q(-4) & \Q(-4) & 0\\
6 &(\bigwedge^2\Lambda)(-2)^{\oplus 2} & (\bigwedge^2\Lambda)(-2) & (\bigwedge^2\Lambda)(-2) \\
4 &\Q(-2)\oplus \Lambda^{\otimes2}(-1)\oplus \Sym^2(\bigwedge^2\Lambda) & \Q(-2)\oplus \Sym^2(\bigwedge^2\Lambda)  &  \Lambda^{\otimes2}(-1)\\
2 &\bigwedge^2\Lambda^{\oplus 2}&\bigwedge^2\Lambda&\bigwedge^2\Lambda\\
0 &\Q & \Q & 0\\
\end{array}
$$

\caption{\label{SxS} Cohomology of $S\times S/(i_1,i_2)$}
\end{table}

\begin{lem}\label{cohofibreg}
The $(i_1,i_2,\iota)$-invariant part of the Leray spectral sequence associated with the $\C^*$-bundle $g^{-1}([S])\rightarrow S\times S$ gives rise to a spectral sequence $E_\pu^{p,q}\Rightarrow \coh{(k_2\circ k_3|_{\rho(U)})^{-1}([S])}{p+q}$ which behaves as follows:
\begin{itemize}
\item[-]
$E_2^{p,q}$ vanishes for $q\neq0,1$;
\item[-]
$E_2^{p,0}$ is the part of $\coh{S\times S}k$ which is invariant under $i_1$, $i_2$ and $\kappa$; 
\item[-]
$E_2^{p,1}$ is the part of $\coh{S\times S}k$ which is invariant under $i_1$, $i_2$ and alternating under $\kappa$, tensored with the Tate Hodge structure $\Q(-1)$.
\end{itemize}
Furthermore, the $E_\infty$ term of this spectral sequence, together with its structure as $\Sp(4,\Q)$-representation, is as given in Table~\ref{fibreoverS}.  
\end{lem}

\begin{table}
$$\begin{array}{r|cccccccccc}
q&&&&&&&&&\\[6pt] 
1&0     &0&\bigwedge^2\Lambda(-1)    &0&  \Lambda^{\otimes2}(-2)            &0&(\bigwedge^2\Lambda)(-3) &0& 0\\
0&\Q&0&(\bigwedge^2\Lambda)&0&\Q(-2)\oplus \Sym^2(\bigwedge^2\Lambda)&0&  (\bigwedge^2\Lambda)(-2)&0&\Q(-4)\\
\hline
&0&1&2&3&4&5&6&7&8&\mil p\mil
\end{array}
$$
\caption{\label{fibreoverS_E2}$E_2$ term of the spectral sequence converging to $H^k(g^{-1}([S]);\Q)^{(i_1,i_2,\iota)}$}
\end{table}

\begin{table}
$$\begin{array}{r|cccccccccc}
q&&&&&&&&&\\[6pt]
1&0     &0&0    &0& \V_{2,0}(-2)            &0&\V_{1,1}(-3) &0& 0\\
0&\Q&0&\Q(-1)\oplus\V_{1,1}&0&\Q(-2)^{\oplus2}\oplus\V_{2,2}&0&  0 &0&0\\
\hline
&0&1&2&3&4&5&6&7&8&p
\end{array}
$$
\caption{\label{fibreoverS}$E_\infty$ term of the spectral sequence converging to $H^k(g^{-1}([S]);\Q)^{(i_1,i_2,\iota)}$}
\end{table}

\begin{proof}
Let us consider the $\C^*$-bundle
$f_S:=f|_{g^{-1}([S])}\co g^{-1}([S])\rightarrow S\times S$. The Leray spectral sequence in cohomology associated with $f_S$ converges to the cohomology of $g^{-1}([S])$ and has $E_2$ term $E_2^{p,q}\cong\coh {S\times S}p\otimes\coh {\C^*}q$.
However, we are only interested in the part of the cohomology of $g^{-1}([S])$ which is invariant under $j_1, j_2$ and $\iota$. Since the actions of $j_1$, $j_2$ and $\iota$ respect the map $g^{-1}([S])\rightarrow S\times S$, they act also on  the terms of the Leray spectral sequence associated with $f_S$. In particular, the spectral sequence whose $E_r$ terms are the $(j_1,j_2,\iota)$-invariant part of the terms of the Leray spectral sequence associated with $f_S$ converges to ${\coh{g^{-1}([S])}k}^{j_1,j_2,\iota}$.

In particular, the $E_2$ term of this spectral sequence is given by the $(j_1,j_2,\iota)$-invariant part of $\coh {S\times S}p\otimes\coh {\C^*}q$.
 We have already determined the behaviour of the projection of these involutions to $S\times S$ in Table~\ref{SxS}, so it remains only to determine their action on the fibre $\C^*$. Since $j_1$ and $j_2$ both fix the fibre of $f$ over the origin, they act trivially on the cohomology of $\C^*$. Instead, the Kummer involution $\iota$ acts as the identity on $H^0(\C^*;\Q)$ and as the alternating representation on $H^1(\C^*;\Q)$. From this the first part of the claim follows. 
For the convenience of the reader, we have written the $E_2$ term of the spectral sequence in Table~\ref{fibreoverS_E2}. Notice that this spectral sequence has only two non-trivial rows. Therefore, it could be written equivalently as a long exact sequence. In particular, the only differentials one needs to study are the $d_2$-differentials.

These differentials are given by restriction of the differentials of the Leray spectral sequence associated with the $\C^*$-bundle $f_S$. Recall that $f_S$ is the $\C^*$-bundle obtaining by subtracting the $0$-section from the Poincar\'e bundle over $S\times S$. Therefore (see e.g. \cite[XVI.7.5]{Husem}) the $d_2$-differentials are given by taking the intersection product with the first Chern class of the Poincar\'e bundle, which is known to be equal to $[\op{diag}(S)]-[S\times \{0\}]-[\{0\}\times S]$, where $[\cdot]$ denotes the fundamental class and $\op{diag}\co S \rightarrow S\times S$ is the diagonal map. 
An explicit computation of the intersections of this class with the $\kappa$-alternating classes in ${\coh{S\times S}k}^{i_1,i_2}$ yields the description of $E_3=E_\infty$ given in Table~\ref{fibreoverS}.
Here we have used the fact that $\Sym^2\Lambda$ is the irreducible $\Sp(4,\Q)$ -representation $\V_{2,0}$, whereas $\bigwedge^2\Lambda$ decomposes into irreducible $\Sp(4,\Q)$-representations as $\Q(-1)\oplus \V_{1,1}$ and $\Sym^2(\bigwedge^2\Lambda)$ decomposes as $\Q(-2)^{\oplus2}\oplus \V_{1,1}(-1)\oplus \V_{2,2}$. In the notation, Tate twists are only relevant for the Hodge structure.

\end{proof}

\subsection{Geometry of $\rho(\Delta)$}\label{discrim}

The map $\rho$ identifies the two components of $\Delta$, each of which is isomorphic to $\cX_2 \times_{\ab2} \cX_2$. In particular, the space $\rho(\Delta)$ can be realized as a finite quotient of $\cX_2 \times_{\ab2} \cX_2$.
This can be read off from the
construction of the toroidal compactification, as in \cite[Lemma 2.4]{S-B}. 
See also \cite[Section I]{HS1} for an outline of this
construction. Also note that the stratum $\Delta$ corresponds to
the stratum in the partial compactification in the direction of the $2$-dimensional cusp associated with a
maximal-dimensional cone in the second Voronoi decomposition for $g=4$. A detailed
description can be found in \cite[Part I, Chapter 3]{HKW}.

Specifically, the stratum $\rho(\Delta)$ corresponds to the $\GL(2,\Z)$-orbit of the cone $\langle x_1^2,x_2^2,(x_1-x_2)^2\rangle$. 
Hence, the map $\xb2\times_{\ab2}\xb2\rightarrow\rho(\Delta)$ is the quotient map with respect to the stabilizer $G$ of the cone $\langle x_1^2,x_2^2,(x_1-x_2)^2\rangle$ in $\Sym^2(\Z^2)$. This is generated by three involutions: the multiplication map by $-1$, the involution interchanging $x_1$ and $x_2$ and the involution
$x_1\mapsto x_1,\ x_2\mapsto x_1-x_2.$

These generators of $G$ act on $\xb2\times_{\ab2}\xb2$  by the following three involutions: the involution $i_1$ which acts by $(x,y)\mapsto (-x,-y)$ on each fibre $S \times S$,
the involution $i_2$ which interchanges
the two factors of $\cX_2 \times_{\ab2} \cX_2$
 and finally
the involution $i_3$ which acts by $(x,y) \mapsto (x+y,-y)$. 

From this description, it follows that there is a fibration $g'\co\rho(\Delta)\rightarrow \ab2$ whose fibre over $[S]\in\ab2$ is isomorphic to the quotient of $S\times S$ by the subgroup of $\Aut(S\times S)$ generated by the three involutions $i_1$, $i_2$ and $i_3$
introduced above. 

If we write $\Lambda:=\coh{S\times \{0\}}1$ and $\Lambda':=\coh{\{0\}\times S}1$, then the cohomology of $S\times S$ is the exterior algebra of $\coh{S\times S}1 = \Lambda \oplus \Lambda'$. If we denote by $f_1,\dots,f_4$, resp. $f_5,\dots,f_8$ the generators of $\Lambda$, resp. $\Lambda'$, the three involutions act of $H^1(S\times S;\Q)$ as follows:
\begin{equation}
f_i \mapsto -f_i,\ i=1,\dots,8,
\end{equation}
\begin{equation}
f_i\leftrightarrow f_{i+4},\ i=1,\dots,4,
\end{equation}
\begin{equation}
f_i\mapsto f_i,\ f_{i+4}\mapsto f_i-f_{i+4},\ i=1,\dots,4.
\end{equation}

Then one proceeds to determine the invariant part of the exterior algebra of $\Lambda\oplus\Lambda'$ under these involution. Moreover, to determine the local systems $R^q_!g'_*(\Q)$ that appear in the Leray spectral sequence associated with $g'\co\rho(\Delta)\rightarrow \ab2$, one needs to investigate the structure of the invariant subspaces as representations of $\Sp(4,\Q)$. An explicit calculation of the invariant classes yields the results which we summarize in the following lemma.
\begin{lem}\label{fibrediscr}
The rational cohomology groups the fibre of $g'\co \rho(\Delta)\rightarrow\ab2$ over a point $[S]\in\ab2$, with their mixed Hodge structures and structure as $\Sp(4,\Q)$-representations, are given by
$$
\coh{{g'}^{-1}([S])}k=\left\{
\begin{array}{ll}
\Q&k=0,\\ 
\bigwedge^2\V_{1,0} = \Q(-1)\oplus \V_{1,1}&k=2,\\
\Q(-2)^{\oplus2}\oplus\V_{1,1}(-1)\oplus \V_{2,2}&k=4,\\
(\bigwedge^2\V_{1,0})(-2) = \Q(-3)\oplus \V_{1,1}(-2)&k=6,\\
\Q(-4)&k=8,\\
0 & \text{otherwise.}
\end{array}
\right.
$$

\end{lem}

\subsection{Proof of Proposition~\ref{cohom_beta2}}\label{proof-beta2}
We will prove Proposition~\ref{cohom_beta2} by investigating the Leray spectral sequence associated with the fibration $k_2\circ k_3\co \beta^0_2\rightarrow\ab2$. As explained at the beginning of this section, the fibre of $k_2\circ k_3$ over a point $[S]\in\ab2$ is the disjoint union of an open part, which is 
$(k_2\circ k_3|_{\rho(U)})^{-1}([S])$,
and a closed part, which is the fibre of $g'\co\rho(\Delta)\rightarrow \ab2$. The cohomology of the fibre of $k_2\circ k_3|_{\rho(U)}$ was determined in Lemma~\ref{cohofibreg}, whereas the cohomology of the fibre of $g'$ was computed in Lemma~\ref{fibrediscr}. 
Notice that $(k_2\circ k_3|_{\rho(U)})^{-1}([S])=g^{-1}([S])/(j_1,j_2,\iota)$ is the finite quotient of a smooth quasi-projective variety, so that we can use Poincar\'e duality to obtain its cohomology with compact support from its cohomology. Furthermore, since $g'^{-1}([S])=S^2/(i_1,i_2,i_3)$ is compact, its cohomology with compact support coincides with its cohomology.
 
To compute the cohomology with compact support of the fibre of $k_2\circ k_3$ one can use the Gysin long exact sequence associated with the inclusion $g'^{-1}([S])\hookrightarrow (k_2\circ k_3)^{-1}([S])$:
\begin{multline}\label{gysin_fibre}
\cdots
\rightarrow
{\cohc{g^{-1}([S])}k}^{(j_1,j_2,\iota)}
\rightarrow
\cohc{(k_2\circ k_3)^{-1}([S])}k
\rightarrow\\
{\cohc{S\times S}k}^{(i_1,i_2,i_3)}
\xrightarrow{\delta_k}
{\cohc{g^{-1}([S])}{k+1}}^{(j_1,j_2,\iota)}
\rightarrow
\cdots
\end{multline}

Notice that all differentials $\delta_k$ in \eqref{gysin_fibre} have to respect the structure of the cohomology groups as representations of $\Sp(4,\Q)$. In this specific case, this implies that all $\delta_k$ with $k\neq 2$ vanish, whereas 
$$\delta_2\co \Q(-1)\oplus \V_{1,1}\longrightarrow \V_{1,1}$$ 
is surjective by Lemma~\ref{diff-rank2} below.

The above determines the cohomology with compact support of the fibre of $k_2\circ k_3$. 
In particular, it also determines the local systems $R^q_!(k_2\circ k_3)_*(\Q)$ occurring in the Leray spectral sequence in cohomology with compact support associated with the fibration $k_2\circ k_3$. These local systems are given in the first column of Table~\ref{leray_beta2}. 

\begin{table}
$$
\begin{array}{r@{}r|ccccc}
\text{$R^q_!(k_2\circ k_3)_*(\Q)$\ }&q&&&&&\\[6pt]
{\scriptstyle
\Q(-5)
\ \ }
&
 10&0&\Q(-7)&0&\Q(-8)\\
{\scriptstyle 0\ \ }&
  9&0&0&0&0\\ 
{\scriptstyle
\Q(-4)^{\oplus2}\oplus \V_{1,1}(-3)
\ \ }
&
  8&\Q(-3)&\Q(-6)^{\oplus2}&0&\Q(-7)^{\oplus2}\\
{\scriptstyle 0\ \ }&
  7&0&0&0&0\\
{\scriptstyle
\Q(-3)^{\oplus3}\oplus\V_{1,1}(-2)%
\ \ }
& 
  6&\Q(-2)\oplus H(-1)&\Q(-5)^{\oplus3}\oplus H(-1)&0&\Q(-6)^{\oplus3}\\
{\scriptstyle
\oplus \V_{2,2}(-1)
}
& \\
{\scriptstyle
\V_{2,0}(-1)
\ \ }
&
  5&\Q(-2)&0&0&0\\
{\scriptstyle
\Q(-2)^{\oplus2}
\oplus \V_{1,1}(-1)\ \ %
}
&
  4&\Q(-1)\oplus H&\Q(-4)^{\oplus2}\oplus H&0&\Q(-5)^{\oplus2}\\
{\scriptstyle
\oplus \V_{2,2}
}
&
\\
{\scriptstyle 0\ \ }&
  3&0&0&0&0\\
{\scriptstyle
\Q(-1)
\ \ }
&
  2&0&\Q(-3)&0&\Q(-4)\\
{\scriptstyle 0\ \ }&
  1&0&0&0&0\\ \ 
{\scriptstyle
\Q
\ \ }
&
  0&0&\Q(-2)&0&\Q(-3)\\ 
\hline
&&3&4&5&6&p
\end{array}
$$
\caption{\label{leray_beta2} $E_2$ term of the Leray spectral sequence converging to the cohomology with compact support of $\beta^0_2$. We denote $H=\cohc[\V_{2,2}]{\ab2}3\cong\cohc[\V_{2,2}]{\ab2}4$ (up to grading).}
\end{table}

Recall that the $E_2$ term of the Leray spectral sequence $E_r^{p,q}\Rightarrow \cohc{\beta^0_2}{p+q}$ associated with $k_2\circ k_3$ are of the form $E_2^{p,q}=\cohc[R^q_!(k_2\circ k_3)_*(\Q)]{\ab2}p$. From the decomposition into symplectic local systems of the $R^q_!(k_2\circ k_3)_*(\Q)$, one gets the $E_2$ term of the Leray spectral sequence as in Table~\ref{leray_beta2}.
Here we used the description of the cohomology with compact support of $\ab2$ with coefficients in the local systems $\V_{1,1}$, $\V_{2,0}$ and $\V_{2,2}$ from Lemma~\ref{ab2-wt2} and \ref{ab2-v22}. %

To prove the claim, it remains to show that the Leray spectral sequence degenerates at $E_2$. From the shape of the spectral sequence, it follows that all $d_2$ differentials, and all differentials $d_r$ with $r\geq4$  are necessarily trivial. The only differentials one needs to investigate are the $d_3$-differentials $E^3_{3,q}\rightarrow E^3_{6,q-2}$. These are necessarily $0$ by Hodge-theoretic reasons, because morphisms of Hodge structures between pure Hodge structures of different weights are necessarily trivial. From this the claim follows.
\qed

\begin{lem}\label{diff-rank2}
The differential $\delta_2:\;{\cohc{S\times S}2}^{(i_1,i_2,i_3)}
\rightarrow
{\cohc{g^{-1}([S])}3}^{(j_1,j_2,\iota)}$ is surjective.
\end{lem}
\proof
We shall prove the claim by an explicit computation on the generators of the groups involved.
Since in the proofs of Lemma~\ref{cohofibreg} and Lemma~\ref{fibrediscr} we described  the cohomology of the fibres of $E$ rather than those of the cohomology with compact support, to compute the rank of $\delta_2$ we shall compute the rank of the map induced by $\delta_2$ on cohomology by Poincar\'e duality
$$\delta_2^*:\;{\coh{S\times S}6}^{(i_1,i_2,i_3)}\otimes \Q(-1)
\longrightarrow
{\coh {g^{-1}([S])}7}^{(j_1,j_2,\iota)},
$$ 
which can be described explicitly as the composition of the map 
$$
\begin{matrix}
{\coh{S\times S}6}^{(i_1,i_2,i_3)}&\longrightarrow&\coh {g^{-1}([S])}7\\
\alpha&\longmapsto&Q\otimes \alpha,
\end{matrix}
$$
where $Q$ denotes the image of the generator of $\coh{\C^*}1$ inside the cohomology of $g^{-1}([S])$, and the symmetrization with respect to the group $G$ generated by $j_1,j_2$ and $\iota$.

A direct computation yields that the classes
$$v_{i,j,k,l}=f_i\wedge f_j\wedge f_{i+4}\wedge f_{j+4}\wedge(2f_k\wedge f_l+2f_{k+4}\wedge f_{l+4}+f_k\wedge f_{l+4}+f_{k+4}\wedge f_l)$$
with $\{i,j,k,l\}=\{1,2,3,4\}$
form a basis of ${\coh{S\times S}6}^{(i_1,i_2,i_3)}$. Here $f_1,\dots,f_8$ denote the basis of $f_1,\dots,f_8$ described in Section~\ref{discrim}. Then we have
$$
\delta_2^*(v_{i,j,k,l})=Q\otimes f_i\wedge f_j\wedge f_{i+4}\wedge f_{j+4}\wedge(f_k\wedge f_{l+4}+f_{k+4}\wedge f_l)
$$
and these classes generate $\coh {g^{-1}([S])}7$. From this the claim follows.
\qed

\section{Torus rank $3$}\label{s:beta_3}

In this section we compute the cohomology with compact support of the stratum with torus rank $3$. As in the previous section, our strategy is based on a detailed geometric analysis of the fibration $\beta^0_3\rightarrow\ab1$ whose toric part is actually independent of the choice of $g=4$. 

\subsection{Description of the geometry}

We first note that the spaces $\AIGU 4$ and $\AVOR 4$ only differ over $\ab 0 $ and hence $\beta_3^\Perf \setminus \beta_4^\Perf =
\beta_3^\Vor \setminus \beta_4^\Vor =: \beta_3^0 $. In this section we want to compute %
$\cohc{\beta^3_0}\pu$.
For this we first give a geometric description.

In order to compactify $\ab 4$ we start with the lattice $\Z^4$. The choice of a toroidal compactification corresponds to the choice of an
admissible fan $\Sigma_4$ in the cone of semi-positive forms in $\Sym^2(\Z^4)$. 
One possible choice for such a fan is given by the perfect cone decomposition $\Sigma_4^{\Perf}$. 
A cusp of $\ab 4$ corresponds to the choice
of an isotropic subspace $U \subset \Q^4$. In our case, %
 for the stratum over $\ab 1$ we take $U= \langle e_1,e_2,e_3 \rangle $ 
where 
the $e_i$ ($1 \leq i \leq 4$) are the 
standard basis of $\Z^4$. This defines an embedding $\Sym^2(\Z^3) \subset \Sym^2(\Z^4)$ and, by restriction of $\Sigma_4^{\Perf}$, also a fan  
in $\Sym^2(\Z^3)$ which is nothing but $\Sigma_3^{\Perf}$. The stratum $\beta_3^0$ itself consists of different strata which are in 
$1:1$-correspondence with the 
$\GL(3,\Z)$-orbits of the
cones $\sigma$ in $\Sigma_3^{\Perf}$ whose interior contains rank $3$ matrices. Up to the action of $\GL(3,\Z)$
there is a unique minimal cone with
this property, namely the cone $\sigma^{(3)}=\langle x_1^2, x_2^2, x_3^2 \rangle $. Beyond that there are (again up to group action) $4$ further cones.
In dimension $4$ 
there are two cones, namely $\sigma_I^{(4)}= \langle x_1^2, x_2^2, x_3^2, (x_2-x_3)^2 \rangle$ and  
$\sigma_{II}^{(4)}= \langle x_1^2, x_2^2, (x_2 - x_3)^2, (x_1-x_3)^2 \rangle $. In dimensions $5$ and $6$ there are one cone each, namely   
$\sigma^{(5)}=\langle x_1^2, x_2^2, x_3^2, (x_2-x_3)^2, (x_1-x_3)^2 \rangle $ and $\sigma^{(6)}=\langle x_1^2, x_2^2, x_3^2, (x_2-x_3)^2, (x_1-x_3)^2, (x_1-x_2)^2\rangle $.
Note that all cones are contained in $\sigma^{(6)}$. In fact the perfect cone decomposition in genus $3$ (where it coincides with the 
second Voronoi decomposition) is obtained by taking the $\GL(3,\Z)$-orbit of $\sigma^{(6)}$ and all its faces. 

To describe the various strata let $\cX_1 \to \ab 1 $ be the universal elliptic curve and let 
$\cX_1 \times_{\ab 1} \cX_1 \times_{\ab 1} \cX_1\to \ab 1 $ be
the triple product with itself over $\ab 1$. Let $T=\Sym^2(\Z^3) \otimes \C^*$ be the $6$-dimensional torus associated with $\Sym^2(\Z^3)$.
Every cone $\sigma$ in $\Sigma_3^{\Perf}$ is basic (i.e. the generators of the rays are part of a $\Z$-basis of $\Sym^2(\Z^3)$)
and defines a subtorus $T^{\sigma} \subset T$ of rank $\dim(\sigma)$. We can now give a description of $\beta_3^0$.

\begin{prop} The variety $\beta_3^0$ admits a stratification into strata as follows:
\begin{itemize}
\item[(i)] there are $6$ strata of $\beta_3^0$, corresponding to the cones $\sigma^{(3)}$, $\sigma_I^{(4)}$, $\sigma_{II}^{(4)}$,
$\sigma^{(5)}$ and $\sigma^{(6)}$.
\item[(ii)] Each stratum is the finite quotient of a torus bundle over $\cX_1 \times_{\ab 1} \cX_1 \times_{\ab 1} \cX_1\to \ab 1 $ with fibre
$T/T^{\sigma}$.   
\end{itemize}
\end{prop}

\begin{proof}
See \cite[Lemma 2.4]{S-B}.
\end{proof}

We shall now compute the cohomology with compact support for each of these strata and then use a spectral sequence argument
to compute the cohomology with compact support of $\beta_3^0$. We denote the substratum of $\beta_3^0$ associated with a cone $\sigma$
by $\beta_3^0(\sigma)$ and the total space of the torus bundle by $\mathcal T(\sigma)$.

Before we state the results we have to give a brief outline of the construction of the stratum $\beta_3^0(\sigma)$ with a view
towards describing suitable coordinates in which our calculations can be done. Consider a point in 
Siegel space of genus $4$:

$$
\tau=
\begin{pmatrix}
\tau_{1,1}&\tau_{1,2}&\tau_{1,3}&\tau_{1,4}\\
\tau_{1,2}&\tau_{2,2}&\tau_{2,3}&\tau_{2,4}\\
\tau_{1,3}&\tau_{2,3}&\tau_{3,3}&\tau_{3,4}\\
\tau_{1,4}&\tau_{2,4}&\tau_{3,4}&\tau_{4,4}\\
\end{pmatrix}
\in \H_4.
$$
Going to the cusp over $\ab 1$ means sending the top left hand $3 \times 3$ block of this matrix to $i \infty$.
We shall make this more precise.
We consider the basis of $\Sym^2(\Z^3)$ given by $U_{i,j}^*=(2 - \delta_{i,j})x_ix_j$. 
Let $t_{i,j}$ ($1\leq i,j\leq 3$) be the dual basis. Setting
$$t_{i,j}=e^{2\pi \sqrt{-1} \tau_{i,j}}\ (1\leq i,j\leq 3)$$
defines a map
\begin{equation}\label{equ_partquot}
\H_4 \to T \times \C^3 \times \H_1
\end{equation}
\begin{equation*}
\tau \mapsto ((t_{i,j}), \tau_{1,4}, \tau_{2,4}, \tau_{3,4}, \tau_{4,4}).
\end{equation*}
This corresponds to taking the partial quotient $X(U)=P'(U) \backslash \H_4$ with respect to the center $P'(U)$ of the 
unipotent radical  of the parabolic subgroup
$P(U)$ associated with the cusp $U$. We denote $P''(U)=P(U) / P'(U)$.
The partial quotient $X(U)$ can be considered as an open set of the trivial torus bundle ${\mathcal X}(U)$
(with fibre $T$) over
$\C^3 \times \H_1 $. Using the fan $\Sigma_3^{\Perf}$ one constructs ${\mathcal X}_{\Sigma_3^{\Perf}}(U)$ by taking a fibrewise toric embedding. 
Let $X_{\Sigma_3^{\Perf}}(U)$ be the interior of the closure of $X(U)$ in  ${\mathcal X}_{\Sigma_3^{\Perf}}(U)$. 
The action of the group $P''(U)$ on $X(U)$ extends to an action on $X_{\Sigma_3^{\Perf}}(U)$
and one obtains the partial 
compactification in the direction of the cusp $U$ by $Y_{\Sigma_3^{\Perf}}(U)=P''(U) \backslash X_{\Sigma_3^{\Perf}}(U)$.

Every cone $\sigma \in  \Sigma_3^{\Perf}$ defines an affine toric variety $X_{\sigma}$. Since all cones $\sigma$ are basic
one has $X_{\sigma}=\C^{k} \times (\C^*)^{6-k}$ where $k$ is the number of generators of $\sigma$. Every inclusion
$\sigma \subset \sigma'$ induces an inclusion $X_{\sigma} \subset X_{\sigma'}$. Note that $X_{(0)}=T$ and, in particular
we obtain an inclusion $X_{(0)}=T \subset X_{\sigma^{(6)}} \cong \C^6$. Let $T_1, \ldots ,T_6$ be the coordinates on
$X_{\sigma^{(6)}} \cong \C^6$ corresponding to the generators of $\sigma^{(6)}$ which form a basis of $\Sym^2(\Z^3)$. 
Computing the dual basis of this basis one finds that this inclusion is given by
\begin{equation}\label{equ_coordinates}
\begin{array}{lllll}
T_1=t_{1,1}t_{1,3}t_{1,2},&&
T_2=t_{2,2}t_{2,3}t_{1,2},&&
T_3=t_{3,3}t_{1,3}t_{2,3},\\
T_4=t_{2,3}^{-1},&&
T_5=t_{1,3}^{-1},&&
T_6=t_{1,2}^{-1}.
\end{array}
\end{equation}

The relation to the strata $\beta_3^0(\sigma)$ is then the following. The coordinate $\tau_{4,4}$ defines
a point in $\ab 1$ and the coordinates $\tau_{1,4},\tau_{2,4},\tau_{3,4}$ define a point in the fibre of 
$\cX_1 \times_{\ab 1} \cX_1 \times_{\ab 1} \cX_1\to \ab 1 $ over $[\tau_{4,4}] \in \ab 1 $ which is 
$E_{\tau_{4,4}} \times E_{\tau_{4,4}} \times  E_{\tau_{4,4}}$, where $E_{\tau_{4,4}}=\C/(\Z+\Z\tau_{4,4})$ is the elliptic curve defined by
$\tau_{4,4}$. The fibres of $\beta_3^0(\sigma) \to \cX_1 \times_{\ab 1} \cX_1 \times_{\ab 1} \cX_1$ are isomorphic
to the torus $T/T^{\sigma}$. 

Finally, we have to make some comments on the structure of the parabolic subgroup $P(U)$. This group is generated by four types
of matrices. The first type are block matrices of the form
$$
g_1=
\begin{pmatrix}
{\bf 1}& 0 & S & 0\\
0 & 1 & 0 & 0\\
0 & 0 & {\bf 1} & 0\\
0 & 0 & 0 & 1
\end{pmatrix}, \text{ where }
S={}^tS \in \Sym^2(\Z^3).
$$
These matrices generate the center $P'(U)$ of the unipotent radical and act by 
$$
\begin{pmatrix}
\tau_{1,1}&\tau_{1,2}&\tau_{1,3}&\tau_{1,4}\\
\tau_{1,2}&\tau_{2,2}&\tau_{2,3}&\tau_{2,4}\\
\tau_{1,3}&\tau_{2,3}&\tau_{3,3}&\tau_{3,4}\\
\tau_{1,4}&\tau_{2,4}&\tau_{3,4}&\tau_{4,4}
\end{pmatrix}
\to
\begin{pmatrix}
\tau_{1,1} + s_{1,1}&\tau_{1,2} + s_{1,2}&\tau_{1,3}+ s_{1,3}&\tau_{1,4}\\
\tau_{1,2}+ s_{1,2}&\tau_{2,2} + s_{2,2}&\tau_{2,3}+ s_{2,3}&\tau_{2,4}\\
\tau_{1,3}+ s_{1,3}&\tau_{2,3}+ s_{2,3}&\tau_{3,3}+ s_{3,3}&\tau_{3,4}\\
\tau_{1,4}&\tau_{2,4}&\tau_{3,4}&\tau_{4,4}
\end{pmatrix}
$$
giving rise to the partial quotient $\H_4 \to T \times \C^3 \times \H_1$ described above.

The second set of generators is of the form
$$
g_2=
\begin{pmatrix}
{\bf 1}& 0 & 0 & 0\\
0 & a & 0 & b\\
0 & 0 & {\bf 1} & 0\\
0 & c & 0 & d
\end{pmatrix}, \text{where}
\begin{pmatrix}
a & b\\
c & d
\end{pmatrix} \in \SL(2,\Z),
$$
resp. 
$$
g_3=
\begin{pmatrix}
{\bf 1}& M & 0 & N\\
0 & 1 & {}^tN & 0\\
0 & 0 & {\bf 1} & 0\\
0 & 0 & -{}^tM & 1
\end{pmatrix}, \text{ where }
M,N \in \Z^3.
$$
Note that the elements of type $g_2,g_3$ generate a Jacobi group, which, in particular, acts on the 
base $\C^3 \times \H_1$ of the partial quotient by $P'(U)$ given by the map
$\H_4 \to \C^3 \times \H_1$ giving rise to the triple product $\cX_1 \times_{\ab 1} \cX_1 \times_{\ab 1} \cX_1$.

Finally we have matrices of the form
$$
g_4=
\begin{pmatrix}
{}^tQ^{-1}& 0 & 0 & 0\\
0 & 1 & 0 & 0\\
0 & 0 & Q & 0\\
0 & 0 & 0 & 1
\end{pmatrix}, \text{ where }
Q \in \GL(3,\Z).
$$
These matrices are of particular importance to us as they operate on the space $\Sym^2(\Z^3)$ by 
$$
\GL(3,\Z)\ni g:\ M\mapsto {}^tQ^{-1}MQ^{-1}.
$$

\subsection{The cohomology of $\beta_3^0(\sigma^{(3)})$}\label{sigma3}

In this section we will prove

\begin{lem}\label{coho_sigma3}
The rational cohomology groups with compact support of $\beta_3^0(\sigma^{(3)})$ are given by
$$\cohc {\beta_3^0(\sigma^{(3)})}k=\left\{
\begin{array}{ll}
\Q(-k/2)&k=12,14\\
\Q((5-k)/2) & k=9,7\\
0 & \text{otherwise.}
\end{array}
\right.
$$
\end{lem}

We start by giving an explicit description of the torus bundle $\mathcal T({\sigma^{(3)}})$ defined by the cone ${\sigma^{(3)}}$.

\begin{lem}\label{zerothstep}
Let $q_{\sigma^{(3)}}\co \mathcal T({\sigma^{(3)}})\rightarrow \xb1\times_{\ab1}\xb1\times_{\ab1}\xb1$ be the 
rank $3$ torus bundle associated with ${\sigma^{(3)}}$. 
Then over each fibre $E\times E\times E$ of $\xb1\times_{\ab1} \xb1\times_{\ab1} \xb1$ we have
$$\mathcal T({\sigma^{(3)}})|_{E\times E\times E}\cong p_{2,3}^*(\cP^0)\oplus p_{1,3}^*(\cP^0)\oplus p_{1,2}^*(\cP^0)$$
where $\cP^0$ is the Poincar\'e bundle over the product $E\times E$ with the $0$-section removed, and $p_{i,j}\co E\times E\times E \rightarrow E\times E$ is the projection to the $i$th and $j$th factor. 
\end{lem}

\proof
We first recall the following description of the Poincar\'e bundle over $E \times E$ where $E=\C/(\Z + \Z \tau)$. Consider the action of the group 
$\Z^4$ on the trivial rank-$1$ bundle on $\C \times \C$ given by 
\begin{equation}\label{equ_poincare}
(n_1,n_2,m_1,m_2)\co (z_1,z_2,w) \mapsto 
\end{equation}
$$
(z_1 + n_1 + m_1 \tau, z_2 + n_2 + m_2 \tau, w e^{-2 \pi i (m_1z_2+m_2z_1+m_1m_2\tau)}) 
$$
(where the $z_i$ are the coordinates on the base and $w$ is the fibre coordinate).
We claim that the quotient line bundle on $E\times E$ is the Poincar\'e bundle. For this it is enough to see that this line bundle is 
trivial on $E \times \{ 0 \}$ and  $\{ 0 \} \times E$ (which is obvious) and that it is isomorphic to ${\mathcal O}_E(O - P)$ on $E \times \{ P \}$.
The latter can be checked by comparing the transformation behaviour of (\ref{equ_poincare}) to the transformation behaviour of the 
theta function $\vartheta(z,\tau)$ in one variable (see e.g. \cite[15.1.3.]{la}).

We have to compare this to our situation. In this case we have an action of the group generated by the matrices
$g_3$ with $M,N \in \Z^3$. For $N=(n_1,n_2,n_3)$ we have $\tau_{i,4} \mapsto \tau_{i,4} + n_i$ and for $M=(m_1,m_2,m_3)$ we have
$\tau_{i,j} \mapsto \tau_{i,j}  + m_j\tau_{i,4} + m_i\tau_{j,4} +m_i m_j\tau_{4,4}$ for $1 \leq i,j \leq 3$
and $\tau_{i,4} \mapsto \tau_{i,4} +m_i\tau_{4,4}$.
Recall that the entries $\tau_{i,4}$ for $i=1,2,3$ are 
coordinates on the factors of $E \times E \times E$ and that it follows from (\ref{equ_coordinates}) that 
we can choose $t_{i,j}^{-1}$ with 
$t_{i,j}=e^{2 \pi i \tau_{i,j}}$ for $(i,j)=(1,2),(1,3),(2,3)$ as coordinates on the torus $\mathcal T({\sigma^{(3)}})$.
Comparing this to the transformation (\ref{equ_poincare}) gives the claim.
\qed

\begin{proof}[Proof of Lemma~\ref{coho_sigma3}]
Recall that the stratum ${\beta_3^0}({\sigma^{(3)}})$ is a finite quotient of the rank~$3$ torus bundle $q_{\sigma^{(3)}}\co\mathcal T({\sigma^{(3)}})\rightarrow\xb1\times_{\ab1}\xb1\times_{\ab1}\xb1$.
This enables us to calculate its rational cohomology by exploiting Leray spectral sequences. 

Notice that the base of $q_{\sigma^{(3)}}$ is the total space of the fibration  $p\co \xb1\times_{\ab1}\xb1\times_{\ab1}\xb1\rightarrow \ab1$.
Over a point $[E] \in \ab 1$, the fibre of $p$ is $p^{-1}([E]) \cong E \times E \times E$ and
the fibre of $p \circ q_{\sigma^{(3)}}$ over $[E]$ is the total space of the rank $3$ torus 
bundle $q_{\sigma^{(3)}}|_{E \times E \times E}: \mathcal T(\sigma^{(3)})|_{E \times E \times E} \to E \times E \times E$
described in Lemma~\ref{zerothstep}.
The cohomology of $(p \circ q_{\sigma^{(3)}})^{-1}([E])$ can be computed by the Leray spectral 
sequence associated with this rank $3$ torus bundle:
\begin{multline}\label{lerayfibre}
E_2^{p,q}(q_{\sigma^{(3)}})=\coh{T/T^{({\sigma^{(3)}})}}q\otimes\coh{E\times E\times E}p\\
\Longrightarrow \coh{(p \circ q_{\sigma^{(3)}})^{-1}([E])}{p+q}.
\end{multline}

Note that the cohomology of $E\times E\times E$ (respectively, the torus $T/T^{({\sigma^{(3)}})}$) is an exterior algebra generated by $\coh{E\times E\times E}1$ (resp. $\coh{T/T^{({\sigma^{(3)}})}}1$).
We denote by $Q_1$, $Q_2$ and $Q_3$, respectively, the generators of $\coh{T/T^{{\sigma^{(3)}}}}1\cong\coh{(\C^*)^3}1\cong\Q^3$ defined by integrating along the loop around $0$ defined, respectively, by  $|t^{-1}_{2,3}|=1$, $|t^{-1}_{1,3}|=1$ or $|t^{-1}_{1,2}|=1$. 

We can write 
each copy of $E$ as a quotient $E= \C / (\Z e_{2i-1} + \Z e_{2i}); \, i=1,2,3$. 
Then $e_1,\dots,e_6$ give rise to a basis of the first homology group of $E\times E\times E$. We will denote by $f_1,\dots,f_6$ the elements of the basis of $\coh{E\times E\times E}1$ dual to $e_1,\dots,e_6$. Notice that the transformation behaviour of the $f_{2i-1}$ and of the $f_{2i}$ for $1\leq i\leq 3$ agrees with the transformation behaviour of the coordinates $\{\tau_{i,4}|\;1\leq i\leq3\}$ of $\C^3\cong (\Z e_1 + \Z e_2+\dots+\Z e_6)\otimes_\Z \C$
(and that of the differentials $d\tau_{i,4}$ which give rise to classes in cohomology).

As we are interested in the quotient of $\mathcal T({\sigma^{(3)}})$ by the finite group $G({\sigma^{(3)}})$,
we shall compute the invariant cohomology with respect to this group. This is done in Lemma~\ref{firststep} for the 
invariant cohomology of the fibre $\mathcal T({\sigma^{(3)}})|_{E\times E\times E}=(p \circ q_{\sigma^{(3)}})^{-1}([E])$ using 
a Leray spectral sequence argument. 
It remains to determine the local systems $R^i_!(p \circ q_{\sigma^{(3)}})_*(\Q)$ over $\ab1$ 
defined by the fibration $p \circ q_{\sigma^{(3)}}\co{\beta_3^0}({\sigma^{(3)}})=\mathcal T({\sigma^{(3)}})\rightarrow \ab1$. 
This is quite straightforward, since the cohomology with compact support of the fibre is constant in 
degrees $12$ and $10$, and since $\Sym^2\coh E1 $ induces the symplectic local system $\V_{2}$ on $\ab1$.

Recall that the cohomology with compact support of $\ab1$ with constant coefficients is concentrated in degree $2$, and that 
the only non-trivial cohomology group of $\ab1$ with coefficients in $\V_{2}$ is 
$H^1_c(\ab1;\V_{2})=\Q$ (see e.g. \cite[Thm.~5.3]{G-M1n}). 
In particular, it then follows from Lemma~\ref{firststep} that the Leray spectral sequence associated with 
$p \circ q_{\sigma^{(3)}}$ has only two columns 
containing non-trivial $E_2$ terms, so it has to degenerate at $E_2$. From this the claim follows.
\end{proof}

\begin{lem}\label{firststep}
For every $[E]\in\ab1$, the rational cohomology with compact support of the fibre of 
${\beta_3^0}({\sigma^{(3)}})\rightarrow \ab1$, with its Hodge structures, coincides with the $G({\sigma^{(3)}})$-invariant part of the cohomology with compact support of the rank $3$ torus bundle 
$\mathcal T({\sigma^{(3)}})|_{E\times E\times E}$ and is given by
$$\left(\cohc{\mathcal T({\sigma^{(3)}})|_{E\times E\times E}}k\right)^{G({\sigma^{(3)}})}=
\left\{
\begin{array}{ll}
\Q(-6) & k=12,\\
\Q(-5) & k=10,\\
\Sym^2(\coh E1)\otimes\Q(-2)&k=8,\\
\Sym^2(\coh E1)\otimes\Q(-1)&k=6,\\
0&\text{otherwise.}
\end{array}
\right.
$$
\end{lem}

\proof
The stabilizer $G({\sigma^{(3)}})$ of $\sigma^{(3)}$ in $\GL(3,\Z)$ is an extension of the symmetric 
group $\s_3$ (permuting the coordinates $x_1,x_2,x_3$) by $(\Z/2\Z)^3$ (acting by  
involutions $(x_1,x_2,x_3,x_4)\mapsto(\pm x_1,\pm x_2,\pm x_3, x_4)$).

The interchange of two coordinates (say, $x_i$ and $x_j$) acts on $\coh{E\times E\times E}1$ by interchanging $f_{2i-1}$ with $f_{2j-1}$, $f_{2i}$ with $f_{2j}$ and leaving all other generators invariant. The action on $\coh{T/T^{({\sigma^{(3)}})}}1$ interchanges $Q_i$ and $Q_j$ and leaves the third generator invariant.

The automorphism mapping $x_i$ to $-x_i$ acts on $\coh{E\times E\times E}1$ as multiplication by $-1$ on the generators $f_{2i-1}, f_{2i}$ and on $\coh{T/T^{({\sigma^{(3)}})}}1$ as multiplication by $-1$ on $Q_k$ with $k\neq i$. All other generators are invariant.

We can compute the $G({\sigma^{(3)}})$-invariant part of the rational cohomology of the rank $3$ torus bundle $\mathcal T({\sigma^{(3)}})|_{E\times E\times E}$ by restricting to the $G({\sigma^{(3)}})$-invariant part of the Leray spectral sequence \eqref{lerayfibre} associated with $q_{\sigma^{(3)}}$. This yields a spectral sequence $E_2^{p,q}$ converging to the $G({\sigma^{(3)}})$-invariant part of $\coh{\mathcal T({\sigma^{(3)}})|_{E\times E\times E}}{p+q}$. 

A computation of the part of the tensor product $\bigwedge^\pu \coh{E\times E\times E}1\otimes \bigwedge^\pu\coh{T/T^{({\sigma^{(3)}})}}1$ which is invariant under $G({\sigma^{(3)}})$ yields that $E_2^{p,q}$ is non-zero only for $(p,q)\in\{(2,0),(2,1),(4,0),((2,2),(4,1),(6,0)\}$. A precise description of the generators of the non-trivial $E_2$ terms is given in Table~\ref{invariants_sigma3}. Note that the spaces for $p=q=2$ and $p=4,q=2$ are both isomorphic to $\Sym^2\coh E1 $ as $\Sp(2,\Q)$-representations. 

\begin{table}
$$
\begin{array}{cccl}
p&q&\text{dim.}&\text{generators}\\[6pt]
0&0 & 1 & 1\\[6pt]
2&0 & 1 & \sum_{i}f_{2i-1}\wedge f_{2i}\\[6pt]
2&1 & 1 & \sum\limits_{{i< j, k\neq i,j}} Q_k\otimes (f_{2i-1}\wedge f_{2j}-f_{2i}\wedge f_{2j-1})\\[6pt]
2&2  & 3  &\sum\limits_{{i< j, k\neq i,j}}Q_i\wedge Q_j\otimes W_k^{(m)},\ m=1,2,3\\[6pt]
4&0 & 1 & \sum_{i< j}f_{2i-1}\wedge f_{2i}\wedge f_{2j-1}\wedge f_{2j}\\[6pt]
4&1 & 1 & \sum\limits_{{i< j, k\neq i,j}}Q_k\otimes (f_{2i-1}\wedge f_{2j}-f_{2i}\wedge f_{2j-1})\wedge f_{2k-1}\wedge f_{2k} \\[6pt]
4&2  & 3  &\sum\limits_{{i< j, k\neq i,j}}Q_i\wedge Q_j\otimes W_k^{(m)}\wedge f_{2k-1}\wedge f_{2k},\ m=1,2,3\\[6pt]
6&0 & 1 & f_1\wedge f_2\wedge f_3\wedge f_4\wedge f_5\wedge f_6.
\end{array}
$$
\begin{tabular}{m{11cm}}
{\small All indices $i,j,k$ are between $1$ and $3$.  For indices $i<j$ we set $W_k^{(1)} = f_{2i-1}\wedge f_{2j}+f_{2i}\wedge f_{2j-1}$, 
$W_k^{(2)} = f_{2i-1}\wedge f_{2j-1}$
and  
$W_k^{(3)} = f_{2i}\wedge f_{2j}$
for $k\neq i,j$.}
\vspace{5pt}
\end{tabular}
\caption{Description of the generators of the $E_2$ terms of the $G({\sigma^{(3)}})$-invariant part of the spectral sequence associated with $q_{\sigma^{(3)}}$. \label{invariants_sigma3}}
\end{table}

Next, one investigates the differentials of the spectral sequence. 
As differentials have to occur between $E_2^{p,q}$ terms such that the two $p+q$ have different parity, an inspection of the spectral sequence quickly reveals that all differentials have to be trivial, with the possible exception of
\begin{equation}
d_2^{2,1}\co E_2^{2,1}\rightarrow E_2^{4,0}\end{equation}
and
\begin{equation}
d_2^{4,1}\co E_2^{4,1}\rightarrow E_2^{6,0}.\end{equation}

We can determine their rank by exploiting the description of the restriction to $E\times E\times E$ of the torus bundle $\mathcal T({\sigma^{(3)}})$ given in Lemma~\ref{zerothstep} as a direct sum of pull-backs of the Poincar\'e bundle with the $0$-section removed. 
This description implies that one can employ the usual description of $d_2$ 
differentials of $\C^*$-bundles to investigate $d_2^{2,1}$ and $d_2^{4,1}$. 
In particular, each of these differentials is given by formally replacing each 
generator $Q_k$ of $\coh{T/T^{\sigma^{(3)}}}1$ by the first Chern class of the bundle $p_{i,j}^*(\cP)$, 
where $1\leq i<j\leq 3$ are chosen such that  $\{i,j,k\}=\{1,2,3\}$. 
Recall that on the product $E \times E$ the Poincar\'e bundle 
$\cP \cong \cO_{E \times E}(E \times \{0\} +  \{0\} \times E - \Delta)$, where $\Delta$ is the diagonal.
From this one concludes that $c_1(\cP)=f_1 \wedge f_2 + f_3 \wedge f_4 - (f_1 + f_3) \wedge (f_2 + f_4)=
f_2 \wedge f_3 - f_1 \wedge f_4$.
It is then a straightforward calculation to prove that both differentials are isomorphisms.

It remains to pass from cohomology to cohomology with compact support, which we can do by Poincar\'e duality, 
using the fact that  $\mathcal T({\sigma^{(3)}})|_{E\times E\times E}$ is smooth of complex dimension $6$. 
Finally, we can identify the $G({\sigma^{(3)}})$-invariant part of the cohomology with compact support of $\mathcal T({\sigma^{(3)}})|_{E\times E\times E}$ with the cohomology with compact support of its finite quotient $\left(\mathcal T({\sigma^{(3)}})|_{E\times E\times E}\right)/G({\sigma^{(3)}})$, which coincides with the fibre of ${\beta_3^0}({\sigma^{(3)}})\rightarrow \ab1$ over $[E]$.
\qed

\subsection{The cohomology of $\beta_3^0(\sigma_I^{(4)})$}\label{coho_1}

In this section we will prove

\begin{lem}\label{lem_coho_1}
The rational cohomology groups with compact support of $\beta_3^0(\sigma_I^{(4)})$ are given by
$$
H_c^k(\beta_3^0(\sigma_I^{(4)});\Q)=\left\{\begin{array}{ll}
\Q(-6) & k=12\\
\Q(-5)^{\oplus2} & k=10\\
\Q(-4) + \Q(-2) & k=8\\
\Q & k=5\\
0 & \text{otherwise}.\\
\end{array}\right.
$$
\end{lem}

\begin{proof}
We shall make again use of the twofold fibre structure of this stratum. The 
stratum $\beta_3^0(\sigma_I^{(4)})$ is a finite quotient of a rank $2$ torus 
bundle $q_{\sigma_I^{(4)}}\co \mathcal T(\sigma_I^{(4)}) \to \cX_1 \times_{\ab 1} \cX_1 \times_{\ab 1} \cX_1$ 
with fibres isomorphic to
$T/T^{(\sigma_I^{(4)})}$. 
Note that the generators of $\sigma_I^{(4)}$ correspond to the first four generators of
the cone $\sigma^{(6)}$. Comparing this to the embedding described in (\ref{equ_coordinates}) we find that we can choose 
$t^{-1}_{1,3},t^{-1}_{1,2}$ as coordinates on $T/T^{(\sigma_I^{(4)})}$. As before we 
denote $p\co \cX_1 \times_{\ab 1} \cX_1 \times_{\ab 1} \cX_1 \to \ab 1$.

As we are interested in the quotient of $\mathcal T(\sigma_I^{(4)})$ by the finite group $G(\sigma_I^{(4)})$,
we shall compute the invariant cohomology with respect to this group. 
Thus we first have to describe the automorphism group 
$G(\sigma_I^{(4)})$ of the cone $\sigma_I^{(4)}$, i.e. all elements of the 
form $g_3 \in \GL(3,\Z)$ which
fix this cone. We have already discussed this in \cite[Section 3]{HT}. The result is that the automorphism group is 
generated by the following
four transformations:
\begin{equation}\label{equ_trafo1}
x_1 \mapsto x_1, \quad x_2 \mapsto x_2 - x_3, \quad x_3 \mapsto -x_3
\end{equation}
\begin{equation}\label{equ_trafo2}
x_1 \mapsto -x_1, \quad x_2,x_3 \mapsto x_2,x_3
\end{equation}
\begin{equation}\label{equ_trafo3}
x_1 \mapsto x_1, \quad x_2 \leftrightarrow x_3. 
\end{equation}
\begin{equation}\label{equ_trafo4}
x_i \mapsto -x_i; \quad i=1,2,3.
\end{equation}

Note that these automorphisms act trivially on the base of the fibration $\cX_1 \times_{\ab 1} \cX_1 \times_{\ab 1} \cX_1 \to \ab 1$. 

Again we shall determine the invariant 
cohomology of the fibre $(q_{\sigma_I^{(4)}} \circ p)^{-1}([E])$ using the Leray spectral sequence
with terms $E_2^{p,q}=H^q(T/T^{(\sigma_I^{(4)})},\Q) \otimes H^p(E \times E \times E,\Q)$. The result is given by:
\begin{lem}\label{lem:invariantcohomologysigma_I^4}
For every $[E]\in\ab1$, the rational cohomology with compact support of the fibre of 
${\beta_3^0}({\sigma^{(4)}_I})\rightarrow \ab1$, with its Hodge structures, is given by
$$\left(\cohc{\mathcal T({\sigma_I^{(4)}})|_{E\times E\times E}}k\right)^{G({\sigma_I^{(4)}})}=
\left\{
\begin{array}{ll}
\Q(-5) & k=10,\\
\Q(-4)^{\oplus2} & k=8\\
\Sym^2(\coh E1)\otimes\Q(-2)&k=7,\\
\Q(-3) & k=6\\
\Sym^2(\coh E1)&k=4,\\
0&\text{otherwise.}
\end{array}
\right.
$$
\end{lem}
\begin{proof}
We denote %
the generators of $H^1(T/T^{(\sigma_I^{(4)})};\Q) \cong H^1((\C^*)^2;\Q) \cong \Q^2$ corresponding to $t^{-1}_{1,3},t^{-1}_{1,2}$
by $Q_2,Q_3$. The %
$f_i, i=1, \ldots , 6$ are, as before, a basis of the cohomology of the triple 
product $E \times E \times E$.

We must now compute the action on (co)homology of the automorphisms of  $\sigma_I^{(4)}$. 
As a non-trivial example 
we shall do this in detail
in the case of the transformation given in (\ref{equ_trafo1}), the computations in the other cases are analogous. 

The action of this transformation on Siegel space is given by:
$$
\begin{pmatrix}
1 & 0 & 0 & 0\\
0 & 1 & 0 & 0\\
0 & -1 & -1 & 0\\
0 & 0 & 0 & 1
\end{pmatrix}
\begin{pmatrix}
\tau_{1,1}&\tau_{1,2}&\tau_{1,3}&\tau_{1,4}\\
\tau_{1,2}&\tau_{2,2}&\tau_{2,3}&\tau_{2,4}\\
\tau_{1,3}&\tau_{2,3}&\tau_{3,3}&\tau_{3,4}\\
\tau_{1,4}&\tau_{2,4}&\tau_{3,4}&\tau_{4,4}
\end{pmatrix}
\begin{pmatrix}
1 & 0 & 0 & 0\\
0 & 1 & -1 & 0\\
0 & 0 & -1 & 0\\
0 & 0 & 0 & 1
\end{pmatrix} =
$$
$$
= \begin{pmatrix}
\tau_{1,1}&\tau_{1,2}& -\tau_{1,2}-\tau_{1,3}&\tau_{1,4}\\
\tau_{1,2}&\tau_{2,2}& -\tau_{2,2} - \tau_{2,3}&\tau_{2,4}\\
-\tau_{1,2}-\tau_{1,3}&-\tau_{2,2} - \tau_{2,3} &\tau_{2,2} + 2 \tau_{2,3} + \tau_{3,3}& - \tau_{2,4} -\tau_{3,4}\\
\tau_{1,4}&\tau_{2,4}&- \tau_{2,4} -\tau_{3,4}&\tau_{4,4}
\end{pmatrix}.
$$
From this we conclude that under this transformation:
\begin{equation}\label{equ_trafo1'}
Q_2 \mapsto -Q_2 - Q_3; \quad Q_3 \mapsto Q_3; 
\end{equation}
\begin{equation*}
f_i \mapsto f_i, \quad i=1, \ldots ,4 \quad f_i \mapsto -f_{i-2} - f_{i}; \quad i=5,6,
\end{equation*}
Note that the latter coincides with the transformation behaviour of the differentials $d\tau_{i,4}, i=1,2,3$, and the former with the transformation behaviour of $-\tau_{1,3},-\tau_{1,2}$.

An analogous computation for the other automorphisms gives the following results: 
\begin{equation}\label{equ_trafo2'}
Q_2, Q_3 \mapsto -Q_2, -Q_3;  
\end{equation}
\begin{equation*}
f_1, f_2 \mapsto -f_1, -f_2, \quad f_i \mapsto f_i, \quad i=3,\ldots,6. 
\end{equation*}
\begin{equation}\label{equ_trafo3'}
Q_2, Q_3 \mapsto Q_3, Q_2;  
\end{equation}
\begin{equation*}
f_1, f_2 \mapsto f_1, f_2, \quad f_3 \leftrightarrow f_5, \quad f_4 \leftrightarrow f_6, 
\end{equation*}
\begin{equation}\label{equ_trafo4'}
Q_2, Q_3 \mapsto Q_2, Q_3;  
\end{equation}
\begin{equation*}
f_i \mapsto -f_i, \quad i=1,\ldots,6. 
\end{equation*}

Now we must compute the invariant cohomology with respect to $G(\sigma_I^{(4)})$. This can either be done by a (lengthy)
computation by hand or a standard computer algebra system. 

The $E^{p,0}_2$ terms of the spectral sequence can be computed as follows. 
The invariant part $E^{2,0}_2$ of the cohomology group $\coh{T/T^{\sigma^{(4)}_I}}0\otimes\coh{E\times E\times E}2$ is two-dimensional and generated by the tensors
$$
I_1=f_1 \wedge f_2, \ 
I_2=2(f_3 \wedge f_4 + f_5 \wedge f_6) + (f_3 \wedge f_6 + f_5 \wedge f_4).
$$
The term $E^{4,0}_2$ is also two-dimensional, with generators $I_1\wedge I_2$ and $I_2\wedge I_2$. The terms $E^{0,0}_2\cong\coh{T/T^{\sigma^{(4)}_I}}0\otimes\coh{E\times E\times E}0$ and $E^{6,0}_2\cong\coh{T/T^{\sigma^{(4)}_I}}0\otimes\coh{E\times E\times E}6$ are one-dimensional and generated by fundamental classes.  

The term $E^{2,1}_2$, which is the invariant part of $\coh{T/T^{\sigma^{(4)}_I}}1\otimes\coh{E\times E\times E}2$ is $4$-dimensional, with generators
$$
g_{i,j} = \left((Q_2+2Q_3)\otimes f_{j+2}+(2Q_2+Q_3)\otimes f_{j+4}\right)\wedge f_i,\ i,j=1,2.
$$
In particular, it is isomorphic to $\coh E 1\otimes\coh E1 = \Sym^2(\coh E1)\oplus\bigwedge^2\coh E1$. The term $E^{4,1}_2$ is also four-dimensional and generated by $(g_{i,j}\wedge I_2)$. All other $E^{p,1}_2$ vanish.

Finally, the only non-trivial terms of the form $E^{p,2}_2$ are those with $p=2$ and $p=4$. The subspace $E^{2,2}_2\subset \coh{T/T^{\sigma^{(4)}_I}}2\otimes\coh{E\times E\times E}2$ is isomorphic to $\Sym^2(\coh E1)$ and is generated by the invariant tensors
$$
Q_2\wedge Q_3 \otimes f_3 \wedge f_5, \ 
Q_2\wedge Q_3 \otimes (f_3 \wedge f_6 + f_4 \wedge f_5), \ 
Q_2 \wedge Q_3 \otimes f_4 \wedge f_6.
$$
Finally, the subspace $E^{4,2}_2$ is $4$-dimensional and equal to  $E^{2,2}_2\wedge I_1$.

In terms of local systems this gives rise to the Table \ref{t:rank3_sigma_I}.
\begin{table}\caption{\label{t:rank3_sigma_I} $E_2$ term of the spectral sequence converging to the cohomology of $\beta_3^0(\sigma_{I}^{(4)})$}
$$
\begin{array}{r|cccccccc}
q&&&&&&&&\\[6pt]
2&0&0&\V_{2}(-2)&0&\V_{2}(-3)&0&0&\\
1&0&0&\V_{2}(-1) \oplus \Q(-2)&0&\V_{2}(-2)\oplus \Q(-3)&0&0&\\
0&\Q&0&\Q(-1)^{\oplus2}&0&\Q(-2)^{\oplus2}&0&\Q(-3)&
\\\hline
& 0&1&2&3&4&5&6&p
\end{array}
$$
\end{table}
We claim that that the differentials $d^{p,q}_2: E_2^{p,q} \to E_2^{p+2,q-1}$ for $(p,q) =(2,1), (2,2)$ and $(4,1)$ are
of maximal rank. Indeed, by Schur's lemma it is enough to prove that they are non-zero. To check this it is enough to recall
that the torus bundle is isomorphic to $p_{1,3}^*(\cP^0)\oplus p_{1,2}^*(\cP^0)$. In particular, for every class $\alpha\in E^{p,1}_2$ we obtain $d_2^{p,1}(\alpha)$ by replacing $Q_2$ with $c_1(p_{1,3}^*(\cP))=-(f_1\wedge f_6+f_5\wedge f_2)$ and $Q_3$ with $c_1(p_{1,3}^*(\cP))=-(f_1\wedge f_4+f_3\wedge f_2)$ in the expression of $\alpha$. Analogously, for every class $\beta\in E^{2,2}_2$ we get $d_2^{2,2}(\beta)$ by replacing $Q_2\wedge Q_3$ with $Q_2\otimes c_1(p_{1,2}^*(\cP)) - Q_3\otimes c_1(p_{1,3}^*(\cP))$.
Then the claim follows from a straightforward calculation.
\end{proof}

To complete the proof of Lemma \ref{lem_coho_1} is now an easy consequence of 
the Leray spectral sequence of the fibration 
$p \circ q_{\sigma_I^{(4)}}: \mathcal T(\sigma_I^{(4)}) \to \ab 1$. Looking at the weights of the Hodge 
structures, we see immediately that all differentials must vanish and thus the result follows.
\end{proof}

\subsection{The cohomology of $\beta_3^0(\sigma_{II}^{(4)})$}

Before we can describe the cohomology of this stratum we must identify the toric bundle $\mathcal T{(\sigma_{II}^{(4)})}$.
\begin{lem}\label{lem:torussigma_II^4}
Let $p_{1,2} : E \times E \times E \to E \times E$ be the projection onto the first two factors and let 
$q: E \times E \times E \to E \times E$ be the map given by $q(x,y,z)=(x+y+z,z)$. Then 
$$
\mathcal T({\sigma_{II}^{(4)}})|_{E\times E\times E}\cong p_{1,2}^*(\cP^0) \oplus q^*(({\cP^{-1})^0}).
$$
where $\cP^0$ is the Poincar\'e bundle over the product $E\times E$ with the $0$-section removed. 
\end{lem}

\begin{proof}
Since the generators of the cone $\sigma_{II}^{(4)}$ correspond to
$T_1,T_2,T_4,T_5$ we can take $T_6=t_{1,2}^{-1}$ and $T_3=t_{3,3} t_{1,3}t_{2,3}$ as coordinates on the torus
$T/T^{(\sigma_{II}^{(4)})}$. In Lemma~\ref{firststep} we had seen that
the action of the group generated by the matrices
$g_3$ with $M,N \in \Z^3$ is as follows. For $N=(n_1,n_2,n_3)$ we have $\tau_{i,4} \mapsto \tau_{i,4} + n_i$ and for $M=(m_1,m_2,m_3)$ we have
$\tau_{i,j} \mapsto \tau_{i,j}  + m_j\tau_{i,4} + m_i\tau_{j,4} +m_i m_j\tau_{4,4}$ for $1 \leq i,j \leq 3$
and $\tau_{i,4} \mapsto \tau_{i,4}+m_i\tau_{4,4}$. In particular
$$
\tau_{1,2} \mapsto \tau_{1,2}  + m_2\tau_{1,4} + m_1\tau_{2,4} +m_1 m_2\tau_{4,4}
$$
whereas
$$
(\tau_{1,3} + \tau_{2,3} + \tau_{3,3}) \mapsto (\tau_{1,3} + \tau_{2,3} + \tau_{3,3}) +
$$
$$
m_3(\tau_{1,4} + \tau_{2,4} + \tau_{3,4}) + (m_1+m_2+m_3)\tau_{3,4} + m_3(m_1+m_2+m_3)\tau_{4,4}. 
$$
A comparison with the transformation behaviour for the Poincar\'e bundle described in Lemma \ref{firststep} gives the claim.
\end{proof}

\begin{lem}\label{lem_coho_2}
The rational cohomology groups with compact support of $\beta_3^0(\sigma_{II}^{(4)})$ are given by
$$
H_c^k(\beta_3^0(\sigma_{II}^{(4)});\Q)=\Q(-k/2), \quad k=10,12.
$$
\end{lem}
\begin{proof}
As in the previous case we first have to describe the automorphism $G(\sigma_{II}^{(4)})$ of the cone $\sigma_{II}^{(4)}$. 
This group is the symmetric group $S_4$ permuting the generators of the cone together with the map $x_i \mapsto -x_i$.
Hence we can work with the following generators:
\begin{equation}\label{equ_trafoII1}
x_i \mapsto - x_i, \quad i= 1, \dots ,6 
\end{equation}
\begin{equation}\label{equ_trafoII2}
x_1 \leftrightarrow x_2, \quad x_3 \mapsto x_1 + x_2 - x_3
\end{equation}
\begin{equation}\label{equ_trafoII3}
x_1 \mapsto x_1 - x_3, \quad x_2 \mapsto -x_2, \quad x_3 \mapsto -x_3 
\end{equation}
\begin{equation}\label{equ_trafoII4}
x_1 \mapsto x_3 - x_2, \quad x_2 \mapsto -x_2, \quad x_3 \mapsto x_1-x_2. 
\end{equation}
We now have to compute the induced action of these automorphisms on the cohomology groups
$\coh{T/T^{(\sigma_{II}^{(4)})}}\pu \otimes \coh{E \times E \times E}\pu$. To this end, we denote by $Q_3$ (respectively, $R$) the generator of $\coh{T/T^{(\sigma_{II}^{(4)})}}1$ corresponding to the parameter $T_6=t_{1,2}^{-1}$ (respectively, to $T_3=t_{1,3}t_{2,3}t_{3,3}$). 

It is immediately clear that in the case of (\ref{equ_trafoII1}) the action is given by
\begin{equation}\label{equ_trafoII1'}
Q_3, R \mapsto Q_3, R;  
\end{equation}
\begin{equation*}
f_i \mapsto -f_i, \quad i=1,\ldots,6. 
\end{equation*}
We note that this implies that there can be no non-trivial invariant cohomology classes involving terms of odd degree
in $H^{\bullet}(E \times E \times E)$.
Next we claim that the action on cohomology of (\ref{equ_trafoII2}) is given by
\begin{equation}\label{equ_trafoII2'}
Q_3\mapsto Q_3-R, \  R \mapsto -R
\end{equation}
\begin{equation*}
f_i \mapsto f_{i+2}+f_{i+4}, \quad i=1,2; \quad f_i \mapsto f_{i-2}+f_{i+2}, \quad i=3,4; \quad f_i \mapsto -f_i, \quad i=5,6. 
\end{equation*}
To see this we compute
$$
\begin{pmatrix}
0 & 1 & 1 & 0\\
1 & 0 & 1 & 0\\
0 & 0 & -1 & 0\\
0 & 0 & 0 & 1
\end{pmatrix}
\begin{pmatrix}
\tau_{1,1}&\tau_{1,2}&\tau_{1,3}&\tau_{1,4}\\
\tau_{1,2}&\tau_{2,2}&\tau_{2,3}&\tau_{2,4}\\
\tau_{1,3}&\tau_{2,3}&\tau_{3,3}&\tau_{3,4}\\
\tau_{1,4}&\tau_{2,4}&\tau_{3,4}&\tau_{4,4}
\end{pmatrix}
\begin{pmatrix}
0 & 1 & 0 & 0\\
1 & 0 & 0 & 0\\
1 & 1 & -1 & 0\\
0 & 0 & 0 & 1
\end{pmatrix} =
$$
$$
= \begin{pmatrix}
*&\tau_{1,2}+\tau_{2,3}+\tau_{1,3}+\tau_{3,3}& -\tau_{2,3}-\tau_{3,3}&\tau_{2,4}+\tau_{3,4}\\
*&*& -\tau_{1,3} - \tau_{3,3}&\tau_{1,4}+\tau_{3,4}\\
*&*&\tau_{3,3}& -\tau_{3,4}\\
*&*& *&\tau_{4,4}
\end{pmatrix}.
$$
This immediately gives the claim for the $f_i$. For $Q_3,R$ we observe that the action induced on the homology is dual to the action on the subspace $\langle -\tau_{1,2},\tau_{1,3}+\tau_{2,3}+\tau_{3,3}$. Since cohomology is dual to homology, the action on $Q_3,R$ agrees with that on $-\tau_{1,2},\tau_{1,3}+\tau_{2,3}+\tau_{3,3}$.

A similar calculation gives the following results in the remaining cases:
\begin{equation}\label{equ_trafoII3'}
Q_3\mapsto -Q_3, \ R\mapsto -Q_3+R;
\end{equation}
\begin{equation*}
f_i \mapsto f_i, \quad i=1,2; \quad f_i \mapsto -f_i, \quad i=3,4; \quad f_i \mapsto -f_{i-4}-f_i, \quad i=5,6. 
\end{equation*}
\begin{equation}\label{equ_trafoII4'}
Q_3 \leftrightarrow R ;  
\end{equation}
\begin{equation*}
f_i \mapsto f_{i+4}, \quad i=1,2; \quad f_i \mapsto -f_{i-2}-f_i-f_{i+2}, \quad i=3,4; \quad f_i \mapsto f_{i-4}, \quad i=5,6. 
\end{equation*}
It is now straightforward to compute the invariants under $G(\sigma_{II}^{(4)})$. In the cohomology group
$\coh{T/T^{(\sigma_{II}^{(4)})}}{0} \otimes \coh{E \times E \times E}{2}$ we find one invariant tensor, namely
$$
I_1= 
3(f_1 \wedge f_2 + f_3 \wedge f_4)+2\phi + \psi,
$$
where we denoted $\phi = (f_1 + f_3 + f_5) \wedge f_6 + f_5 \wedge (f_2 + f_4 + f_6)$ and $\psi = f_1 \wedge f_4 + f_3 \wedge f_2$.
In  $\coh{T/T^{(\sigma_{II}^{(4)})}}{1} \otimes \coh{E \times E \times E}{2}$ we also obtain one invariant tensor, namely
$$
I_2= -R\otimes(2\phi+\psi) + Q\otimes (\phi+2\psi).$$
The invariant class in $\coh{T/T^{(\sigma_{II}^{(4)})}}{0} \otimes \coh{E \times E \times E}{4}$ is $I_1 \wedge I_1$
and in %
$\coh{T/T^{(\sigma_{II}^{(4)})}}{1} \otimes \coh{E \times E \times E}{4}$
it is $I_2\wedge I_1$.
This together with the fundamental classes in  %
$\coh{T/T^{(\sigma_{II}^{(4)})}}{0} \otimes \coh{E \times E \times E}{0}$
and $\coh{T/T^{(\sigma_{II}^{(4)})}}{0} \otimes \coh{E \times E \times E}{6}$ are the only invariants.

As before we now look at the Leray spectral sequence in cohomology associated with 
$p \circ q_{\sigma_{II}^{(4)}}\co {\mathcal T}(\sigma_{II}^{(4)}) \to \ab1 $.
Since all representations are trivial we 
thus obtain Table {\ref{t:rank3}}.
\begin{table}\caption{\label{t:rank3} $E_2$ term of the spectral sequence converging to the cohomology of $\beta_3^0(\sigma_{II}^{(4)})$}
$$
\begin{array}{r|cccccccc}
q&&&&&&&&\\[6pt]
2&0&0&0&0&0&0&0&\\
1&0&0&\Q(-2)&0&\Q(-3)&0&0&\\
0&\Q&0&\Q(-1)&0&\Q(-2)&0&\Q(-3)&
\\\hline
& 0&1&2&3&4&5&6&p
\end{array}
$$
\end{table}
Hence, we have two differentials which could be non-zero, namely
$$
d_2^{2,1}\co
\coh{T/T^{(\sigma_{II}^{(4)})}}{1} \otimes \coh{E \times E \times E}{2} \to
\coh{T/T^{(\sigma_{II}^{(4)})}}{0} \otimes \coh{E \times E \times E}{4},
$$
resp.
$$
d_2^{4,1}\co \coh{T/T^{(\sigma_{II}^{(4)})}}{1} \otimes \coh{E \times E \times E}{4} \to
\coh{T/T^{(\sigma_{II}^{(4)})}}{0} \otimes \coh{E \times E \times E}{6}.
$$
Indeed we claim that they do not vanish. For this we use the description of 
$T({\sigma_{II}^{(4)}})|_{E\times E\times E}$ given in Lemma \ref{lem:torussigma_II^4}. 
It follows from this description that this bundle splits into the product of two factors with Euler classes 
$-\phi$
and $\psi$.
The claim that the first differential is non-zero is now equivalent to
$$
\phi\wedge(2\phi+\psi)+\psi\wedge(\phi+2\psi)\neq0.
$$
For the second differential we must check that %
$$\left(\phi\wedge(2\phi+\psi)+\psi\wedge(\phi+2\psi)\right)\wedge I_1\neq0.$$
This can be checked by direct calculation. At the same time this proves that the first differential does not vanish.
The claim of the lemma now follows immediately after converting to cohomology with compact support.  
\end{proof}

\subsection{The cohomology of $\beta_3^0(\sigma^{(5)})$}

\begin{lem}\label{lem_coho_3}
The rational cohomology groups with compact support of $\beta_3^0(\sigma^{(5)})$ are given by
$$
H_c^k(\beta_3^0(\sigma^{(5)});\Q)=\left\{\begin{array}{ll}
\Q(-k/2) & k=6,10\\
\Q(-k/2)^{\oplus2} & k=8\\
0 & \text{otherwise}.\\
\end{array}\right.
$$
\end{lem}
\begin{proof}
We first have to compute the automorphism group $G(\sigma^{(5)})$. It is not hard to see that this group is 
generated by the transformations
\begin{equation}\label{equ_trafo5}
x_i \mapsto -x_i, \quad i=1,2,3
\end{equation}
\begin{equation}\label{equ_trafo6}
x_1 \leftrightarrow x_2, \quad x_3 \mapsto x_3
\end{equation}
\begin{equation}\label{equ_trafo7}
x_1 \mapsto x_1 - x_3, \quad x_2 \mapsto x_2 - x_3, \quad x_3 \mapsto -x_3 
\end{equation}
\begin{equation}\label{equ_trafo8}
x_1 \mapsto x_1 - x_3, \quad x_2 \mapsto -x_2, \quad x_3 \mapsto -x_3. 
\end{equation}
A computation analogous to that in Lemma \ref{lem_coho_1} shows that this results in the following action on cohomology,
where again we denoted by $f_i$ the generators of the cohomology of $E\times E\times E$ and by $Q_3$ the generator of the cohomology of the fibre of the torus bundle:

\begin{equation}\label{equ_trafo9}
f_i \mapsto -f_i, i=1, \ldots ,6; \quad Q_3\mapsto Q_3
\end{equation}
\begin{equation}\label{equ_trafo10}
f_i \leftrightarrow f_{i+2},  i=1,2; \quad f_j \mapsto f_j, j=5,6; \quad Q_3\mapsto Q_3
\end{equation}
\begin{equation}\label{equ_trafo11}
f_i \mapsto f_i, i=1, \ldots ,4; \quad f_k \mapsto -f_{k-4}-f_{k-2}-f_k, k=5,6; \quad Q_3\mapsto Q_3
\end{equation}
\begin{equation}\label{equ_trafo12}
f_i \mapsto f_i, i=1,2; \quad f_j \mapsto -f_j, j=3,4; \quad f_k \mapsto -f_{k-4} - f_k, k=5,6; \quad Q_3\mapsto -Q_3.
\end{equation}

Next, we compute the invariant cohomology in $\coh{\C^*}0 \otimes \coh{E \times E \times E}{2k}$. 
Clearly this is $1$-dimensional
for $k=0,6$. By duality it is enough to do the computation for $k=2$. Here we find a $2$-dimensional invariant subspace generated by
$i_1:=f_1 \wedge f_2 +  f_3 \wedge f_4$ and $i_2:=f_1 \wedge f_4 +  f_3 \wedge f_2 + 2 (f_1+f_3+f_5)\wedge f_6 + 2f_5\wedge (f_2+f_4+f_6)$.

In this situation we also have invariant cohomology in 
$\coh{\C^*}1 \otimes \coh{E \times E \times E}{2}$.
This is $1$-dimensional and generated by $Q_3\otimes (f_1\wedge f_4 + f_3 \wedge f_2)$.
By duality we also have a $1$-dimensional invariant subspace in  $H^1(\C^*) \otimes H^4(E \times E \times E)$. A standard calculation shows that
this is generated by 
$Q_3\wedge(f_1 \wedge f_4 +  f_3 \wedge f_2)\wedge i_2$.

In this case the differentials in the Leray spectral sequence are not automatically $0$. The situation is described in 
Table \ref{t:rank2}. 
\begin{table}\caption{\label{t:rank2} $E_2$ term of the spectral sequence converging to the cohomology of $\beta_3^0(\sigma^{(5)})$}
$$
\begin{array}{r|cccccccc}
q&&&&&&&&\\[6pt]
1&0&0&\Q(-2)&0&\Q(-3)&0&0&\\
0&\Q&0&\Q(-1)^{\oplus2}&0&\Q(-2)^{\oplus2}&0&\Q(-3)&
\\\hline
& 0&1&2&3&4&5&6&p
\end{array}
$$
\end{table}
here are two differentials which we have to consider. These are:
$$
d_2^{2,1}: \coh{\C^*}1 \otimes \coh{E\times E \times E}2 \to \coh{\C^*}0 \otimes \coh{E\times E \times E}4,
$$
resp.
$$
d_2^{4,1}: \coh{\C^*}1 \otimes \coh{E\times E \times E}4 \to \coh{\C^*}0 \otimes \coh{E\times E \times E}6.
$$

We claim that both differentials are non-zero, i.e. they have rank $1$. We first treat $d_2^{1,2}$. 
The differential is given 
by taking the cup-product with the first Chern class of the vector bundle spanned by the torus bundle $\mathcal T(\sigma^{(5)})|_{E \times E \times E}$.
As in previous cases on can see that $\mathcal T(\sigma^{(5)})|_{E \times E \times E} \cong p_{12}^*(\mathcal P^0)$.
This shows that
$$
d_2^{1,2}\co Q_3 \otimes (f_1\wedge f_4 + f_3 \wedge f_2) \mapsto  
$$
$$
(f_1\wedge f_4 + f_3 \wedge f_2) \wedge (f_1\wedge f_4 + f_3 \wedge f_2) =
2 f_1 \wedge f_2 \wedge f_3 \wedge f_4 \neq 0. 
$$
The argument for $d_2^{1,4}$ is analogous.
Finally we use the duality $H^k_c(\beta_3^0(\sigma^{(5)});\Q) = H^{10-k}(\beta_3^0(\sigma^{(5)});\Q)^* \otimes \Q(-5)$ 
(which holds on finite smooth covers) to obtain the claim.

\end{proof}
\subsection{The cohomology of $\beta_3^0(\sigma^{(6)})$}\label{sigma6}
\begin{lem}\label{lem_coho_4}
The rational cohomology groups with compact support of $\beta_3^0(\sigma^{(6)})$ are given by
$$
H_c^k(\beta_3^0(\sigma^{(6)});\Q)=\left\{\begin{array}{ll}
\Q(-k/2) & k=2,4,6,8\\
0 & \text{otherwise}.\\
\end{array}\right.
$$
\end{lem}
\begin{proof}
The proof of this lemma is analogous to the other cases. We first note that the automorphism group of $G(\sigma^{(6)})$
is generated by the symmetric group in three variables permuting the coordinates $x_i$ ($i=1,2,3$) and the transformations \eqref{equ_trafo5} and \eqref{equ_trafo7} already considered in the previous section.
In this case the torus rank is $0$ and hence it suffices to compute the action on the cohomology of the triple product
$E \times E \times E$. In view of the transformation (\ref{equ_trafo5}) there is no invariant in odd degree. By duality it is enough to compute
the invariant cohomology in $\coh{E \times E \times E}2$.
A straightforward calculation shows that this is $1$-dimensional with generator
$2v_1+v_2+v_3$, with $v_1=f_1 \wedge f_2 + f_3 \wedge f_4 +f_5 \wedge f_6$, $v_2=f_1 \wedge f_4 + f_3 \wedge f_6 + f_5 \wedge f_2$
and $v_3=f_1 \wedge f_6 +f_3 \wedge f_2 + f_5 \wedge f_4$.
\end{proof}

\subsection{The cohomology of $\beta_3^0 $}
In this section, we will use the computations on the strata of $\beta_3^0$ to prove the following result.

\begin{prop}\label{cohom_beta3}
The rational cohomology with compact support of $\beta_3^0$ is as follows: the non-zero Betti numbers are
$$
\begin{array}{c|cccccccccc}
i&2&4&5&6&7&8&10&12&14\\\hline
b_i&1&1&1&2&1&4&4&3&1
\end{array}
$$
One has $\cohc{\beta_3^0}7=\Q(-1)$ and $\cohc{\beta_3^0}5=\Q$.
Furthermore all cohomology groups of even degree are algebraic. 
\end{prop}

\proof
We consider the Gysin spectral sequence associated with the stratification of $\beta_3^0$ given by the locally closed strata
$W_1=\beta_3^0(\sigma^{(6)})$, $W_2=\beta_3^0(\sigma^{(5)})$, 
$W_3=\beta_3^0(\sigma_I^{(4)}) \cup \beta_3^0(\sigma_{II}^{(4)})$ and $W_4= \beta_3^0(\sigma^{(3)})$.
We set $Y_p = \overline W_p$.
This is the spectral sequence 
$E_\pu^{p,q}\Rightarrow \cohc{\beta_3^{(0)}}{p+q}$ with $E_1^{p,q}=\cohc{Y_p\setminus Y_{p-1}}{p+q}=
\cohc{W_p}{p+q}$. 
The cohomology with compact support of the strata $W_i$ was computed in the Lemmas \ref{coho_sigma3}, \ref{lem_coho_1}, \ref{lem_coho_2}, \ref{lem_coho_3} and \ref{lem_coho_4}.
In view of these results, the $E_1$ term of the Gysin spectral sequence is as given in Table~\ref{gysin-beta30}.  
\begin{table}
\caption{Gysin spectral sequence converging to the cohomology with compact support of $\beta_3^0$}\label{gysin-beta30}
$$
\begin{array}{r|ccccc}
q&\\[6pt]
10&0&0&0&\Q(-7)\\
9&0&0&\Q(-6)^{\oplus2}&0\\
8&0&\Q(-5)&0& \Q(-6)\\
7&\Q(-4)&0&\Q(-5)^{\oplus3}&0\\
6&0&\Q(-4)^{\oplus2}&0&0\\
5&\Q(-3)&0&\Q(-4) + \Q(-2)&\Q(-2)\\
4&0&\Q(-3)&0& 0\\
3&\Q(-2)&0&0&\Q(-1)\\
2&0&0&\Q&0\\
1&\Q(-1)&0&0&0
\\\hline
&1&2&3&4&p
\end{array}
$$
\end{table}
We consider the differentials 
$d_r^{p,q}\co E_r^{p,q} \to E_r^{p+r,q-r+1}$. Inspection of Table~\ref{gysin-beta30} shows that the only possible
non-zero differential is  $d_1^{3,5}\co E_1^{3,5}=\cohc{\beta^0_3(\sigma^{(4)}_I)}8 \to E_1^{4,5}=\cohc{\beta^0_3(\sigma^{(3)})}9$. We can interpret this differential as arising from the Gysin long exact sequence associated with the inclusion of $\beta^0_3(\sigma^{(4)}_I$ in the partial compactification $\beta^0_3(\sigma^{(3)})\cup\beta^0_3(\sigma^{(4)}_I)$ of $\beta^0_3(\sigma^{(3)})$. Let us denote by $\Psi_E$ the fibre of the fibration $\beta^0_3(\sigma^{(3)})\cup\beta^0_3(\sigma^{(4)}_I)\rightarrow \ab1$ over a point $[E]\in\ab1$. Thanks to the Leray spectral sequence associated with that fibration, all we need to know is that the cohomology with compact support of $\Psi_E$ vanishes in degree $7$. This requires to prove that the differential $$d_7:\;\cohc{(p\circ q_{\sigma^{(4)}_I})^{-1}([E])}7\rightarrow \cohc{(p \circ q_{\sigma^{(3)}})^{-1}([E])}8$$ in the Gysin long exact sequence associated with $(p\circ q_{\sigma^{(4)}_I})^{-1}([E])\subset\Psi_E$ is an isomorphism. 

Since in the proofs of Lemma~\ref{firststep} and Lemma~\ref{lem:invariantcohomologysigma_I^4} we described the generators of the cohomology of the fibres of $E$ rather than those of the cohomology with compact support, we shall analyze the map induced by $d_7$ on cohomology by Poincar\'e duality, whose rank coincides with that of $d_7$.
Let us recall that the inclusion of $\{0\}\times(\C^*)^2$ in $\C\times(\C^*)^2$ induces a Gysin long exact sequence whose differentials define the maps
$$
\begin{matrix}
\coh{(\C^*)^2}k\otimes\Q(-1)&\longrightarrow&\coh{(\C^*)^3}{k+1}& &\ \\
T_i &\longmapsto& T_i\wedge T_1&&i=2,3.
\end{matrix}
$$

As a consequence, the differential $\coh{(p\circ q_{\sigma^{(4)}_I})^{-1}([E])}3\otimes\Q(-1)\rightarrow \coh{(p \circ q_{\sigma^{(3)}})^{-1}([E])}4$ maps each of the generators $g_{i,j}$ described in the proof of Lemma~\ref{lem:invariantcohomologysigma_I^4} to the class obtained by replacing $T_2$ by $T_2\wedge T_1$ in the expression, and then symmetrizing for the action of the group $G(\sigma^{(3)})$. This yields:
$$
g_{i,j}\longmapsto \frac23 \sum\limits_{{0\leq k,l\leq 2}}Q_{k+1}\wedge Q_{l+1}\otimes f_{2k+i}\wedge f_{2l+j},
$$
hence in particular the differential is surjective. From this the claim follows.

\qed

\section{Torus rank 4}

In this section we compute the cohomology of the closed strata $\beta_4 \subset \AVOR 4$ and 
$\beta_4^\Perf\subset\APERF 4$ 
of torus rank $4$
in the second Voronoi and the perfect cone compactification, respectively. 

We shall first state the main results:
\begin{thm}\label{Igusa-beta4}
The cohomology groups with rational coefficients of the closed stratum $\beta_4^\Perf\subset\APERF 4$ of the 
perfect cone compactification of the moduli space of abelian varieties of dimension $4$ are non-zero only in even degree.
The only non-zero Betti numbers are 
$b_0=b_2=b_4=1$, $b_6 = b_8=4$, $b_{10}=3$ and $b_{12}=1$.

The cohomology is algebraic in all degrees different from $6$, whereas $\coh{\beta_4^\Perf}6$ is an extension of $\Q(-3)^{\oplus3}$ by $\Q(-1)$. 
\end{thm}

The closed stratum $\beta_4\subset\AVOR 4$ has two irreducible components: a nine-dimen\-sion\-al component $E$, which is the exceptional divisor of the blow-up $q\co \AVOR 4 \rightarrow \APERF 4$, and a six-dimensional component, which is the proper transform of $\beta_4^\Perf$ under $q$.

\begin{thm}\label{Voronoi-beta4}
\begin{enumerate}
\item \label{exceptional}
The rational cohomology of $E$ is all algebraic. The only non-zero Betti numbers are $b_0=b_{2}=b_{16}=b_{18}=1$, $b_4=b_{14}=2$ and $b_6=b_8=b_{10}=b_{12}=3$.
\item \label{all} The rational cohomology of $\beta_4$ is non-trivial only in even degree. The non-zero Betti numbers are
$$
\begin{array}{c|cccccccccc}
i&0&2&4&6&8&10&12&14&16&18\\\hline
b_i&1&2&3&7&7&6&4&2&1&1
\end{array}
$$
All cohomology groups are algebraic, with the exception of $\coh{\beta_4}6$, which is an extension of $\Q(-3)^{\oplus6}$ by $\Q(-1)$.
\end{enumerate}
\end{thm}

\subsection{Cone decompositions}

It is in this section that we require full information about the perfect cone or first Voronoi and the second
Voronoi decomposition in $\Sym_{\geq0}^2(\R^4)$. 
Details concerning these decompositions
can be found in \cite{ER1}, \cite{ER2}, \cite{Vallentin} and \cite{V2b}. We start by recalling the perfect cone
decomposition. The starting point is two $10$-dimensional cones, namely
the principal cone $\Pi_1(4)$ and the second perfect cone $\Pi_2(4)$. These cones are given by
$$
\begin{array}{l@{\;}l}
\Pi_1(4)=&\langle x_1^2, x_2^2, x_3^2, x_4^2,(x_1-x_2)^2,(x_1-x_3)^2,(x_1-x_4)^2,(x_2-x_3)^2,\\
&(x_2-x_4)^2,(x_3-x_4)^2 \rangle
\end{array}
$$
and
$$
\begin{array}{l@{\;}l}
\Pi_2(4)=&\langle x_1^2, x_2^2, x_3^2, x_4^2, (x_1-x_3)^2, (x_1-x_4)^2, (x_2-x_3)^2, (x_2-x_4)^2,\\
&(x_3-x_4)^2, (x_1+x_2-x_3)^2, (x_1+x_2-x_4)^2, (x_1+x_2-x_3-x_4)^2\rangle\\
\end{array}
$$
respectively. The perfect cone decomposition consists of all $\GL(4,\Z)$-translates of these cones and their
faces. While the cone $\Pi_1(4)$ is basic, the cone $\Pi_2(4)$ is not, hence it defines a singular 
point $P_{\sing}\in\APERF 4$. Nevertheless, all $9$-dimensional faces of $\Pi_2(4)$ are basic. 
Modulo the action of $\GL(4,\Z)$ these $9$-dimensional faces define two orbits. Traditionally these are called 
RT (red triangle) and BF (black face) respectively (see \cite{ER2}).

In genus $4$ and $5$ (but not in general) the second Voronoi decomposition is a subdivision of the perfect cone 
decomposition. In our case it is the refinement of the perfect cone decomposition obtained 
by adding all cones that arise as spans of the $9$-dimensional faces of $\Pi_2(4)$ with the central ray generated by 
\begin{multline}\label{ray}
e=\frac 13\left[
x_1^2+x_2^2+x_3^2+x_4^2+(x_1-x_3)^2+(x_1-x_4)^2+(x_2-x_3)^2+(x_2-x_4)^2\right.\\
\left.+(x_3-x_4)^2+(x_1+x_2-x_3)^2+(x_1+x_2-x_4)^2+(x_1+x_2-x_3-x_4)^2\right].
\end{multline}
In particular, all perfect cones, with the exception of $\Pi_2(4)$, belong to the second Voronoi decomposition. 
Geometrically this means that $\AVOR 4$ is a blow-up of $\APERF 4 = \AIGU 4$ in the singular point $P_{\sing}$. 
Since all cones on the second Voronoi decomposition are basic $\AVOR 4$ is smooth (as a stack). Moreover,
the exceptional divisor $E$ is irreducible and smooth (again as a stack). (For a discussion of this see also
\cite{HS2}.
 
A description of representatives of all $\GL(4,\Z)$-orbits of cones in the second Voronoi, and hence also the 
perfect cone decomposition,
can be found in \cite[Chapter 4]{Vallentin}. For cones with extremal rays spanned by quadratic forms of rank $1$ the list 
is given in \cite[S.4.4.4]{Vallentin}.
Note that in this list $K_5$ denotes the cone $\Pi_2(4)$, and the $9$-dimensional cones $K_5-1$ and $K_{3,3}$ 
correspond to the equivalence classes BF, respectively, RT of \cite{ER2}. 
The remaining cones are listed in  \cite[S.4.4.5]{Vallentin}. The following list gives the number of  
$\GL(4,\Z)$-orbits of cones in each dimension for the two decompositions.
$$
\begin{array}{l|cccccccccc}
\operatorname{dimension}&1&2&3&4&5&6&7&8&9&10\\\hline
\# \, {\operatorname {perfect \, cones}}&1&1&2&3&4&5&4&2&2&2\\\hline
\# \, {\operatorname {second \, Voronoi \, cones}}&2&2&4&7&9&11&11&7&4&3 \\
\end{array}
$$
From this we see that the perfect cone decomposition has $26$ different cones, whereas the second Voronoi decomposition has
$60$ different cones. The lists in \cite{Vallentin} also allow us to write down generators for the extremal rays of 
representatives in all cases.
 
\subsection{Plan for computation}

We briefly recall the structure of $\beta_4$ and $\beta_4^\Perf$ which comes from the toroidal construction. 
More generally, let $\beta_4^{\Sigma}$ be the stratum of any 
admissible fan $\Sigma$ (in our case either the
perfect cone or the second Voronoi fan), then each cone $\sigma \in  \Sigma$ defines a torus orbit $T_{\sigma}$
of dimension $10-k$ where $k$ is the dimension of $\sigma$. Let $G_\sigma\subset \GL(4,\Z)$ be the stabilizer of $\sigma$
with respect to the natural action of $\GL(4,\Z)$ on $\Sym_{\geq0}^2(\R^4)$. Then  $G_\sigma$ acts on $T_{\sigma}$
and $\beta_4^{\Sigma}$ is the disjoint union of the quotients $Z_{\sigma}=T_{\sigma} / G_\sigma$ where $\sigma$ 
runs through a set of representatives of all cones in $\Sigma$ which contain a form of rank $4$ in their interior. 
We then define a stratification by defining $S_p$ as the union of all $Z_{\sigma}$ where $\dim\sigma \geq 10-p$.
In particular, $S_p \setminus S_{p-1}$ is the union of all $Z_{\sigma}$ with $\dim\sigma = 10-p$. 

The Gysin spectral sequence $E_\pu^{p,q}\Rightarrow \cohc{\beta_4^\Perf}{p+q}=\coh{\beta_4^\Perf}{p+q}$ associated with the filtration $S_p$ has $E_1$ term given by
$$
E_1^{p,q}=\cohc{S_p\setminus S_{p-1}}{p+q}.
$$
Since $S_p \setminus S_{p-1}$ is the disjoint union of the $Z_\sigma$ with $\dim\sigma =10-p$ it follows that
$$
\cohc{S_p\setminus S_{p-1}}{\pu} = \bigoplus_{\dim\sigma}\cohc{Z_\sigma}{\pu}.
$$ 

In our situation we have considerably more information. In particular we know that, with the exception of $\Pi_2(4)$, all cones
in
both the perfect cone and the second Voronoi decomposition, are basic. In particular all strata $S_p$ with $p \leq 9$
are locally quotients of a smooth variety by a finite group.
Moreover
$T_\sigma =(\C^*)^{10-\dim\sigma}$ and
$Z_\sigma =(\C^*)^{10-\dim\sigma}/G_{\sigma}$. The torus orbit of $\Pi_2(4)$ is a point. Thus we have to compute 
for each cone $\sigma$ the cohomology of the torus $T_{\sigma}$ with respect to $G_{\sigma}$. 
Recall that $\coh {(\C^*)^k}\pu$ is the exterior algebra generated by the $k$-dimensional vector space 
$\coh {(\C^*)^k}1$. Moreover, a basis of the vector space $\coh {(\C^*)^k}1$ can be obtained by taking the Alexander 
dual classes of the fundamental classes of the components 
$$\{(y_1,\dots,y_k)| y_i=0\},\ i=1,\dots,k$$
of the complement of $(\C^*)^k$ in $\C^k$.
This means that, once the generators of the cone $\sigma$ and of the group $G_\sigma$ are known, the computation of 
the cohomology of $Z_\sigma$ reduces to a linear algebra problem, which can be solved using computational tools. 
In our case, the generators of the stabilizers $G_\sigma$ were calculated with Magma (\cite{Magma}) and the 
invariant part of the algebra $\bigwedge^\pu\coh{(\C^*)^k}1$ with Singular (\cite{GPS09}).

\subsection{Perfect cones}

We shall now perform the programme outlined above for the perfect cone compactification, which coincides with the 
Igusa compactification in genus $4$. We have already mentioned that a list of representatives of 
all cones in the perfect cone decomposition, together with their generators, can be found in \cite[Ch.~4]{Vallentin}.
This enables us to compute the stabilizer groups $G_{\sigma}$ as well as the invariant cohomology of the torus
orbits $T_{\sigma}=(\C^*)^k$ where $k=10 - \dim\sigma$.
The results so obtained are listed in Table~\ref{Igusa-strata}, where the notation for the cones is the one 
of \cite[\S4]{Vallentin}. %
The information on the cohomology of the strata is given in the form of Hodge Euler characteristics, i.e. what is given is
the Euler characteristic of $\cohc{Z_\sigma}\pu$ in the Grothendieck group of Hodge structures. 
The symbol $\Ll$ denotes the class of the weight 2 Tate Hodge structure $\Q(-1)$ in the Grothendieck group. 
The relationship between cohomology and cohomology with compact support is given by Poincar\'e duality:
$$\cohc {Z_\sigma}l = \Hom(\coh{Z_\sigma}{2k-l},\Q(-k)),$$
which holds since the $Z_\sigma$ are finite quotients of the smooth varieties $T_\sigma$.

\begin{table}
\caption{\label{Igusa-strata}$\GL(4,\Z)$-orbits of perfect cones}
\begin{tabular}{ll}
$\begin{array}{l@{}cl}
\Sigma & \dim\Sigma & \eul{Z_{\Sigma}}\\[3pt]
K_5=\Pi_1(4) &10&1\\
\Pi_2(4) &10&1\\
K_5-1&9&\Ll\\
K_{3,3}&9&\Ll-1\\
K_5-2&8&\Ll^2\\
K_5-1-1&8&\Ll^2-\Ll\\
K_5-2-1&7&\Ll^3-\Ll^2\\
C_{2221}&7&\Ll^3\\
K_5-3&7&\Ll^3\\
\end{array}$
&
$\begin{array}{l@{}cl}
\Sigma & \dim\Sigma & \eul{Z_\Sigma}\\[3pt]
K_4+1&7&\Ll^3\\
C_{222}&6&\Ll^4-\Ll^3\\
C_{321}&6&\Ll^4+1\\
C_{221}+1&6&\Ll^4\\
C_3+C_3&6&\Ll^4\\
C_5&5&\Ll^5-1\\
C_4+1&5&\Ll^5\\
C_3+1+1&5&\Ll^5+\Ll\\
{1+1+1+1} &4&\Ll^6\\
\end{array}$
\end{tabular}
\end{table}

In view of the information on the cohomology of the $Z_\sigma$ given in Table~\ref{Igusa-strata}, this yields that the $E_1$ terms of the spectral sequence $E_\pu^{p,q}\Rightarrow \coh{\beta_4^\Perf}{p+q}$ are as shown in Table~\ref{Igusa-specseq}. 

\begin{table}
\caption{\label{Igusa-specseq}$E_1$ term of the spectral sequence converging to $\coh{\beta_4^\Perf}\pu$.}
$$
\begin{array}{r|cccccccc}
q&&&&&&&\\[6pt]
6&0&0&0&0&0&0&\Q(-6)\\
5&0&0&0&0&0&\Q(-5)^{\oplus3}\mil\mil&0\\
4&0&0&0&0&\Q(-4)^{\oplus4}\mil\mil&0&0\\
3&0&0&0&\Q(-3)^{\oplus4}\mil\mil&\Q(-3)&0&0\\
2&0&0&\Q(-2)^{\oplus2}\mil\mil&\Q(-2)&0&0&0\\
1&0&\Q(-1)^{\oplus2}\mil\mil&\Q(-1)&0&0&\Q(-1)&0\\
0&\Q^{\oplus2}&\Q&0&0&\Q&\Q&0\\\hline
&0&1&2&3&4&5&6&p
\end{array}
$$
\end{table}

To establish Theorem~\ref{Igusa-beta4}, we need to determine the rank of all differentials in the spectral sequence. As morphisms between pure Hodge structures of different weights are necessarily trivial, one remains with five differentials to investigate, all of the form $d_1^{p,q}\co E_1^{p,q}\rightarrow E_1^{p+1,q}$. We will denote them by 
$$
\begin{array}{lll}
\delta_0\co E_1^{0,0}\longrightarrow E_1^{1,0},&
\delta_0'\co E_1^{4,0}\longrightarrow E_1^{5,0},\\
\delta_1\co E_1^{1,1}\longrightarrow E_1^{2,1},&
\delta_2\co E_1^{2,2}\longrightarrow E_1^{3,2},&
\delta_3\co E_1^{3,3}\longrightarrow E_1^{4,3}.\\
\end{array}
$$

\begin{lem}
All the differentials $\delta_0,\delta_0',\delta_1,\delta_2$ and $\delta_3$ have rank $1$. 
\end{lem}

\begin{proof}
Since $\beta_4^\Perf$ is connected, one has $\coh{\beta_4^\Perf}0=\Q$. This implies that $\delta_0$ has rank $1$.

Next, we consider the differential $\delta_0'\co E_1^{4,0}\cong \Q\longrightarrow E_1^{5,0}\cong \Q$. 
From the description of the strata given in Table~\ref{Igusa-strata}, we have $E_1^{4,0}=\cohc {Z_{C_{321}}}4$ 
and $E_1^{5,0}=\cohc {Z_{C_5}}5$ for the cones
$$\begin{array}{l}
C_{5}=\langle x_1^2, x_2^2. (x_1-x_4)^2, (x_2-x_3)^2, (x_3-x_4)^2\rangle,\\
C_{321}=\langle x_1^2, x_2^2, x_4^2, (x_1-x_4)^2, (x_2-x_3)^2, (x_3-x_4)^2 \rangle.
\end{array}
$$

The cone $C_5$ is contained in $C_{321}$, hence   $Z_{C_{321}}$ is contained in the closure of $Z_{C_5}$. 
Furthermore, the rank of $\delta_0'$ must coincide with the rank of the differential 
$\eta_0\co \cohc {Z_{C_{321}}}4\rightarrow \cohc {Z_{C_{5  }}}5$ of the Gysin long exact sequence associated with 
the inclusion of $Z_{C_{321}}$ in the partial compactification $Z_{C_5}\cup Z_{C_{321}}$ of $Z_{C_5}$. 

If one considers the stabilizers, one observes that $G_{C_{321}}$ is a subgroup of $G_{C_5}$. Therefore, one can view $\eta_0$ as a map from the cohomology of $(\C^*)^4$ to the cohomology of $(\C^*)^5$ in the following way:
$$
\xymatrix@C=6pt@R=12pt{
{\cohc {Z_{C_{321}}}4} \ar@{=}[r]\ar@{->}[d]^{\eta_0} 
& {\cohc {(\C^*)^4}4^{G_{C_{321}}}} \ar@{->}[dr]^{\eta_0}
\\
{\cohc{Z_{C_{5  }}}5}\ar@{=}[r]
&
{\cohc {(\C^*)^5}5^{G_{C_{5}}}}\ar@{=}[r]
&
{\cohc {(\C^*)^5}5^{G_{C_{321}}},}
}
$$
where we used the fact that the $G_{C_{321}}$-invariant part of $\cohc {(\C^*)^5}5$ coincides 
with the $G_{C_5}$-invariant part. This new interpretation relates the map $\eta_0$ to the differential
\begin{equation}\label{diff_delta'}
\cohc {(\C^*)^4}4\cong\Q\longrightarrow \cohc {(\C^*)^5}5\cong\Q
\end{equation}
 of the Gysin exact sequence of an inclusion $(\C^*)^4\hookrightarrow \C\times (\C^*)^4$, with complement 
isomorphic to $(\C^*)^5$. 
In particular, since $\cohc{\C\times (\C^*)^4}k$ vanishes for $k\leq 5$, the differential \eqref{diff_delta'} 
is an isomorphism, and the same holds for $\eta_0$.

Let us consider the differential 
$$\delta_1\co E_1^{1,1}\cong\cohc{Z_{K_5-1}}2\oplus\cohc{Z_{K_{3,3}}}2\rightarrow E_1^{2,1}\cong \cohc{Z_{K_5-1-1}}3.$$
Note that  both $Z_{K_5-1}$ and $Z_{K_{3,3}}$ are contained in the closure
of $Z_{K_5-1-1}\subset\beta_4^\Perf$. 
We choose to investigate the inclusion $i_{3,3}$ of $Z_{K_{3,3}}$ in  the partial compactification 
$Z_{K_{3,3}}\cup Z_{K_5-1-1}$ of $Z_{K_5-1-1}$. 
Then the rank of $\delta_1$ cannot be smaller than the rank of the differential
$$\eta_1\co \cohc{Z_{K_{3,3}}}2\longrightarrow \cohc{Z_{K_5-1-1}}3$$
in the Gysin long exact sequence associated with $i_{3,3}$, even though there is no canonical isomorphism between the kernel of $\eta_1$ and that of $\delta_1$.

In Vallentin's notation, the cone $K_{3,3}$ is given by
$$\begin{array}{ll}
K_{3,3}=&\langle x_1^2, x_2^2, x_3^2, x_4^3, (x_1-x_3)^2, (x_1-x_4)^2, (x_2-x_3)^2, (x_2-x_4)^2, \\
&(x_1+x_2-x_3-x_4)^2 \rangle.
\end{array}
$$
In particular, its subcone
$$K_{5}-1-1b=  \langle x_1^2, x_2^2, x_3^2, x_4^2, (x_1-x_3)^2, (x_1-x_4)^2, (x_2-x_3)^2, (x_2-x_4)^2 \rangle $$
belongs to the same $\GL(4,\Z)$-orbit as $K_{5}-1-1$, so that $Z_{K_{5}-1-1b}\subset\beta_4^\Perf$ 
coincides with $Z_{K_{5}-1-1}$. The stabilizer $G_{K_{5}-1-1b}$ of the cone $K_{5}-1-1b$ is generated by 
$-{\operatorname {Id}}_{\Z^4}$ and by the two automorphisms
$$\left\{\begin{array}{l}
x_1 \mapsto x_3\\
x_2 \mapsto x_4\\
x_3 \mapsto x_2\\
x_4 \mapsto x_1
\end{array}\right.
\text{ and }
\left\{\begin{array}{l}
x_1 \mapsto x_1\\
x_2 \mapsto x_2\\
x_3 \mapsto x_4\\
x_4 \mapsto x_3
\end{array}\right..
$$

In particular, one can check that the group $G_{K_{5}-1-1b}$ is contained in the stabilizer $G_{K_{3,3}}$ of the cone $K_{3,3}$. 

Analogously to the case of $\eta_0$, we can reduce the study of $\eta_1$ to the study of the long exact sequence of an inclusion  $\C^*\hookrightarrow \C\times\C^*$ with complement isomorphic to $(\C^*)^2$, by exploiting the diagram

$$
\xymatrix@C=2pt@R=12pt{
{\cohc {Z_{K_{3,3}}}2} \ar@{->}[d]^{\eta_1} 
\ar@{=}[r]
& {\cohc {\C^*}2^{G_{K_{3,3}}}}
\ar@{=}[r]
& {\cohc {\C^*}2^{G_{K_5-1-1b}}} \ar@<-2ex>[d]^{\eta_1}
\\
{\cohc{Z_{K_{5  }-1-1b}}3}
\ar@{=}[rr]
&&
{\cohc {(\C^*)^2}3^{G_{K_5-1-1b}}.}%
}
$$

Then the claim follows from the fact that the  differential $\cohc{\C^*}2\cong \Q(-1)\rightarrow\cohc{(\C^*)^2}3$ in the Gysin long exact sequence associated with the inclusion $\C^*\hookrightarrow \C\times\C^*$ has rank $1$.

The proof for $\delta_2$ and $\delta_3$ is completely analogous to that for $\delta_1$. In the 
case of $\delta_2$ one considers the inclusion of the $2$-dimensional stratum $Z_{K_5-2}$ in the $3$-dimensional 
stratum $Z_{K_5-2-1}$, given by the inclusion of the cone 
$$
K_5-2-1b = \langle x_1^2, x_2^2, x_3^2, (x_1-x_4)^2, (x_2-x_3)^2, (x_2-x_4)^2, (x_3-x_4)^2 \rangle,$$
which lies in the same $\GL(4,\Z)$-orbit as $K_5-2-1$, in 
$$
K_5-2 = \langle x_1^2, x_2^2, x_3^2, x_4^2, (x_1-x_4)^2, (x_2-x_3)^2, (x_2-x_4)^2, (x_3-x_4)^2\rangle.
$$
In this case, the stabilizers of $K_5-2$ and of $K_5-2-1b$ coincide as subgroups of $\GL(4,\Z)$.

In the case of $\delta_3$, one considers the inclusion of the $3$-dimensional stratum $Z_{C_{2221}}$ in the 
$4$-dimensional stratum $Z_{C_{222}}$, given by the inclusion of the cone 
$$
C_{222}= \langle x_1^2, x_2^2, x_3^2, (x_1-x_4)^2, (x_2-x_4)^2, (x_3-x_4)^2 \rangle
$$
in
$$
C_{2221} = \langle x_1^2, x_2^2, x_3^2, x_4^2, (x_1-x_4)^2, (x_2-x_4)^2, (x_3-x_4)^2 \rangle .
$$
Again, the stabilizers of $C_{222}$ and $C_{2221}$ coincide.
\end{proof}

\subsection{Cones containing $e$}

We shall now prove Theorem~\ref{Voronoi-beta4}. 

\begin{proof}[Proof of \eqref{exceptional}$\Rightarrow$\eqref{all} in Theorem~\ref{Voronoi-beta4}]
Assume that the cohomology with compact support of the exceptional divisor $E$ of the 
blow-up $\AVOR 4 \rightarrow \APERF 4$ is as stated in \eqref{exceptional}. 
The Gysin long exact sequence associated with the closed inclusion $E\subset\beta_4$ is as follows: 
\begin{equation}\label{les_exceptional}
\cdots\rightarrow 
\cohc E{k-1}\xrightarrow{d_k}
\cohc{\beta_4 \setminus E}k
\rightarrow \cohc{\beta_4}k
\rightarrow \cohc{E}k
\rightarrow \cdots
\end{equation}
Similarly the Gysin sequence of the pair $\{P_\sing\} \subset \beta_4^\Perf$ reads
$$%
\cdots\rightarrow 
\cohc {P_{\sing}}{k-1}
\rightarrow \cohc{\beta_4^\Perf\setminus \{P_\sing\}}{k}
\rightarrow
$$
$$
\rightarrow \cohc{{\beta_4}^{\Perf}}k
\rightarrow \cohc {P_{\sing}}{k}
\rightarrow \cdots
$$
Since $\AVOR 4\rightarrow \APERF 4$ is an isomorphism outside $E$, the complement $\beta_4\setminus E$ is 
isomorphic to $\beta_4^\Perf\setminus\{P_\sing\}$.  
By Theorem \ref{Igusa-beta4} the odd cohomology with compact support of $\beta_4^\Perf$ vanishes and
hence $\cohc{\beta_4^\Perf\setminus\{P_\sing\}}k=\cohc{\beta_4 \setminus E}k=0$ for odd $k \geq 3$.
Moreover, $\cohc{\beta_4^\Perf\setminus\{P_\sing\}}1=\cohc{\beta_4 \setminus E}1=0$ since 
$P_\sing$ is a point and $\beta_4^\Perf$ is compact (which implies that cohomology with compact support and 
ordinary cohomology coincide).

Furthermore by the description of $\cohc E\pu$ from \eqref{exceptional} we know that all odd cohomology of $E$ vanishes.
This ensures that all differentials $d_k, k\geq 1$ are zero.
This implies that the Betti numbers $b_k$ of $\beta_4$ with $k\geq 1$ are as stated in Theorem~\ref{Voronoi-beta4}. 
Also the description of the Hodge structures follows from Theorem~\ref{Igusa-beta4} and from part~\eqref{exceptional} 
in view of the long exact sequence \eqref{les_exceptional}.
Finally, the fact that $\cohc{\beta_4}0$ is one-dim\-en\-sion\-al follow from the connectedness of $\beta_4$. 
To complete the proof, recall that $\beta_4$ is compact, so that cohomology and cohomology with compact support agree.
\end{proof}

\begin{lem}\label{purity}
For every $k$, the cohomology group $\coh Ek$ carries a pure Hodge structure of weight $k$.
\end{lem}

\begin{proof}
To prove the claim, we consider the second Voronoi compactification $\AVOR 4(n)$ of the moduli space of 
principally polarized abelian fourfolds with a level-$n$ structure ($n\geq 3$). 
Recall that $\AVOR 4(n)$ is a smooth projective scheme and that the map $\pi(n)\co \AVOR 4(n)\rightarrow \AVOR 4$ 
is a finite group quotient. The preimage $\pi(n)^{-1}(E)$ of $E$ is the union of finitely many irreducible components, 
all of which are smooth and pairwise disjoint. 
This follows from the toric description, since these components are themselves
toric varieties given by the star $\operatorname{Star}(\langle e \rangle)$ in the lattice 
$\Sym^2(\Z^4)/\Z e$.

In particular, this implies that the Hodge structures on the cohomology groups of $\pi(n)^{-1}(E)$ are 
pure of weight equal to the degree.
As $\pi(n)$ is finite, the pull-back map 
$$\pi(n)|_{\pi(n)^{-1}(E)}^*\co \coh Ek \rightarrow \coh {\pi(n)^{-1}(E)}k$$
 is injective. This implies that  each cohomology group $\coh Ek$ is a Hodge substructure of $\coh{\pi(n)^{-1}(E)}k$, thus yielding the claim.   
\end{proof}

\begin{table}
\caption{\label{Voronoi-strata1}$\GL(4,\Z)$-orbits of cones of dimension $\geq 6$ containing $e$}

\begin{tabular}{ll}
$\begin{array}{l@{}cl}
\sigma & \dim\sigma & \eul{Z_{\sigma}}\\[3pt]
111+      &10&1\\
111-&10&1\\

211+      &9&\Ll-1\\
211- &9&\Ll\\
311+      &8&\Ll^2-\Ll\\
311- &8&\Ll^2\\
22'1      &8&\Ll^2-\Ll\\
221+      &8&\Ll^2\\
221-&8&\Ll^2+\Ll\\
411&7&\Ll^3-\Ll^2\\
321+&7&\Ll^3-\Ll^2+\Ll-1
\end{array}$ &
$\begin{array}{l@{}cl}
\sigma & \dim\sigma & \eul{Z_{\sigma}}\\[3pt]
321-&7&\Ll^3-2\Ll^2+\Ll\\
222' &7&\Ll^3-\Ll^2\\
22'2''&7&\Ll^3\\
222+      &7&\Ll^3\\
222-&7&\Ll^3\\

421&6&\Ll^4-\Ll^3+\Ll^2-\Ll\\
331+      &6&\Ll^4+1\\
331-&6&\Ll^4-\Ll^3-\Ll+1\\
322+      &6&\Ll^4-\Ll^3\\
322-      &6&\Ll^4-\Ll^3\\
322'&6&\Ll^4-2\Ll^3+2\Ll^2-2\Ll+1\\
\end{array}$ 

\end{tabular}
\end{table}

\begin{proof}[Proof of \eqref{exceptional} in Theorem~\ref{Voronoi-beta4}]
In view of Lemma~\ref{purity}, determining the cohomology of $E$ is equivalent  to computing its Hodge Euler characteristics, i.e.
$$\eul E=\sum_{k\in\Z}(-1)^k[\cohc Ek],$$
where $[\cdot]$ denotes the class in the Grothendieck group $K_0(\mathsf{HS}_\Q)$ of Hodge structures.
Hodge Euler characteristics are additive, so we are going to work with a locally closed stratification of $E$ and add up the Hodge Euler characteristics to get the result.

The toroidal construction of $\AVOR 4$ yields that $E$ is the union of  toric strata $Z_\sigma$ for all 
cones $\sigma$  belonging to the second Voronoi decomposition but not to the perfect cone decomposition. 
Note that for such $\sigma$ the variety $Z_\sigma$ automatically maps to $\ab 0$ under the 
map $\AVOR 4\rightarrow  \ASAT 4$. Furthermore, up to the action of $\GL(4,\Z)$, one can assume that these 
cones contain the extremal ray  $\la e\ra$ defined in  \eqref{ray} as extremal ray.

\begin{table}
\caption{\label{Voronoi-strata2}$\GL(4,\Z)$-orbits of cones of dimension $\leq 5$ containing $e$}
$\begin{array}{l@{}cl}
\Sigma & \dim\Sigma & \eul{Z_{\sigma}}\\[3pt]
422' &5&\Ll^5+\Ll^3-\Ll^2+\Ll\\
332- &5&\Ll^5-\Ll^4+\Ll^3-3\Ll^2+2\Ll\\
431&5&\Ll^5-\Ll^4+\Ll^3-\Ll^2+\Ll-1\\
422       &5&\Ll^5-\Ll^4\\
332+      &5&\Ll^5-2\Ll^4+\Ll^3-\Ll^2+2\Ll-1\\

432&4&\Ll^6-2\Ll^5+2\Ll^4-4\Ll^3+5\Ll^2-2\Ll\\
333- &4&\Ll^6+2\Ll^2\\
441&4&\Ll^6+\Ll^2\\
333+      &4&\Ll^6-\Ll^5-\Ll^3+2\Ll^2-\Ll\\

433&3&\Ll^7-\Ll^6+\Ll^5-\Ll^4+4\Ll^3-4\Ll^2\\
442&3&\Ll^7+2\Ll^3-\Ll^2\\

443&2&\Ll^8+2\Ll^4-3\Ll^3\\
444&1&\Ll^9-\Ll^4\\\hline
\text{TOT.}^{\phantom{\big\{}}&&\Ll^9+\Ll^8+2\Ll^7+3\Ll^6+3\Ll^5+3\Ll^4+3\Ll^3+2\Ll^2+\Ll+1\\
\end{array}$
\end{table}

Since cones that lie in the same $\GL(4,\Z)$-orbit give the same variety $Z_\sigma$, we have to work with a list of 
representatives of all $\GL(4,\Z)$-orbits of cones fulfilling our conditions. 
Such a list is given in \cite[\S4.4.5]{Vallentin}. As in the proof of Theorem~\ref{Igusa-strata}, 
we compute for each cone $\sigma$ in Vallentin's list the generators of its stabilizer $G_\sigma$ in $\GL(4,\Z)$, 
as well as their action on $\coh{(\C^*)^{10-\dim\sigma}}1$. Then we use the computer algebra program 
Singular \cite{GPS09} to calculate all positive Betti numbers of the quotient $Z_\sigma =(\C^*)^{10-\dim\sigma}/G_\sigma$.
The results are given in Tables~\ref{Voronoi-strata1} and~\ref{Voronoi-strata2}, where we list all 
cones and the Hodge Euler characteristics of the corresponding strata of $E$. 

As already explained, the Hodge Euler characteristic of $E$ is the sum for the Euler characteristics of all strata 
$Z_\sigma$ and is computed at the bottom of Table~\ref{Voronoi-strata2}. In view of Lemma~\ref{purity}, and 
recalling that $\Ll$ is the notation of the weight 2 Tate Hodge structure $\Q(-1)$ in the Grothendieck group of rational
Hodge structures, we can conclude that the Betti numbers of $E$ agree with those given in the statement of 
Theorem~\ref{Voronoi-beta4}. 
\end{proof}

\begin{rem}
Note that the Betti numbers of $E$ satisfy Poincar\'e duality. Indeed, this must be the case as $E$ is smooth up to finite 
group action.
\end{rem}

\appendix
\section{Cohomology of $\ab3$ with coefficients in symplectic local systems}\label{appendix}

In this section, we recollect the information on the cohomology of local systems on $\ab2$ and $\ab3$ that we used in the course of the paper. Let us recall that the cohomology of local systems of odd weight on $\ab g$ vanishes because it is killed by the abelian involution. Therefore, we only need to deal with local systems of even rank.

The cohomology of $\ab2$ and $\ab3$ with constant coefficients is known. The moduli space $\ab2$ is the disjoint union of the moduli space $\M2$ of genus $2$ curves and the locus $\Sym^2{\ab1}$ of products. Since it is known that the rational cohomology of both these spaces vanishes in positive degree, we have
\begin{lem}
The only non-trivial rational cohomology groups with compact support of $\ab2$ are $\cohc{\ab2}4=\Q(-2)$ and $\cohc{\ab2}2=\Q(-1)$.
\end{lem}

The rational cohomology of $\ab3$ was computed by Hain (\cite{Hain}). We state below his result in terms of cohomology with compact support.
\begin{thm}[Hain]\label{hain}
The non-trivial Betti numbers with compact support of $\ab 3$ are
$$
\begin{array}{c|cccc}
i&12&10&8&6\\\hline
b_i&1&1&1&2
\end{array}
$$
Furthermore, all cohomology groups are algebraic with the exception of $\cohc{\ab3}6$, which is an extension of $\Q(-3)$ by $\Q$.
\end{thm}

We deduce the results we need on non-trivial symplectic local systems with weight $\leq 2$ from results on moduli spaces of curves (\cite{BTM4bar},\cite{OTM32}). Note that the result for $\V_{1,1}$ was already proven in \cite[Lemma 3.1]{HT}.

\begin{lem}\label{m2-wt2}
The cohomology groups with compact support of the weight $2$ symplectic local systems on $\M2$ are as follows: the cohomology of $\V_{2,0}$ vanishes in all degrees, whereas the only non-zero cohomology group with compact support of $\V_{1,1}$ is $\cohc[\V_{1,1}]{\M2}3=\Q$.
\end{lem}
\begin{lem}\label{m3-wt2}
The rational cohomology of $\M3$ with coefficients in $\V_{1,0,0}$ and $\V_{2,0,0}$ is $0$ in all degrees.
The only non-trivial cohomology group with compact support of $\M3$ with coefficients in $\V_{1,1,0}$ is $\cohc[\V_{1,1,0}]{\M3}9=\Q(-5)$. 
\end{lem}

\begin{proof}[Proof of Lemma~\ref{m3-wt2}]
Following the approach of \cite{G-2}, we use the forgetful maps $p_1\co \Mm 31\rightarrow \M3$ and $p_2\co \Mm 32\rightarrow \M3$ to obtain information. 
Note that $p_1$ is the universal curve over $\M3$ and that the fibre of $p_2$ is the configuration space of $2$ distinct points on a genus $3$ curve. 

According to \cite[Cor.~1]{BTM4bar}, there is an isomorphism $\coh{\Mm31}\pu\cong\coh{\M3}\pu\otimes\coh{\Pp1}\pu$ as vector spaces with mixed Hodge structures. If we compare this with the Leray spectral sequence in cohomology associated with $p_1$, we get that the cohomology of $\M3$ with coefficients in $\V_{1,0,0}$ must vanish.

Next, we analyze the Leray spectral sequence in cohomology %
 associated with $p_2$. Taking the $\s_2$-action into account, the cohomology of the fibre of $p_2$ induces the following local systems on $\M3$:

$$
\begin{array}{l|l|l}
&\text{local system:}&\text{local system:}\\
\text{deg.}&\text{invariant part}&\text{alternating part}\\\hline
0 &\Q&0\\
1 & \V_{1,0,0} & \V_{1,0,0}\\
2 & \Q(-1)\oplus\V_{1,1,0}&\Q(-1)\oplus\V_{2,0,0}\\
3 & 0&\V_{1,0,0}(-1)
\end{array}
$$

This implies that the cohomology of $\M3$ with coefficients in $\V_{2,0,0}$ (respectively, in $\V_{1,1,0}$) is strictly related to the $\s_2$-alternating (resp. $\s_2$-invariant) part of the cohomology of $\Mm32$. 
The rational cohomology of $\Mm 32$ is described with its mixed Hodge structures and the action of the symmetric group in \cite[Thm~1.1]{OTM32}. By comparing this with the $E_2$-term of the Leray spectral sequence associated with $p_2$, one obtains that the cohomology of $\V_{2,0,0}$ vanishes and that the only non-trivial cohomology group of $\V_{1,1,0}$ is $\coh[\V_{1,1,0}]{\M3}3=\Q(-3)$. Then the claim follows from Poincar\'e duality.
\end{proof}

\begin{proof}[Proof of Lemma~\ref{m2-wt2}]
The proof is analogous to that of Lemma~\ref{m3-wt2}. In this case, one needs to compare the Leray spectral sequence associated with $p_2:\;\Mm22\rightarrow\M2$ with the cohomology of $\Mm22$ computed in \cite[II,2.2]{OT-thesis}. Note that in this case the cohomology of $\V_{1,0}$ vanishes because it is killed by the hyperelliptic involution on the universal curve over $\M2$. 
\end{proof}

Next, we compute the cohomology of the weight $2$ local systems on $\ab2$ and $\ab3$ we are interested in, by using Gysin long exact sequences in cohomology with compact support and the stratification
$\ab2 = \tau_2(\M2)\sqcup \Sym^2{\ab1}$ of $\ab2$, respectively, the stratification
$\ab3 = \tau_3(\M3)\sqcup\tau_2(\M2)\times\ab1 \sqcup\Sym^3\ab1$
of $\ab3$. 
The result on the cohomology with compact support of $\V_{1,1}$ was already proved in \cite[Lemma 3.1]{HT}.

\begin{lem}\label{ab2-wt2}
The only non-trivial cohomology groups of $\ab2$ with coefficients in a local system of weight $2$ are $\cohc[\V_{1,1}]{\ab2}3=\Q$ and $\cohc[\V_{2,0}]{\ab2}3=\Q(-1)$.
\end{lem}

\proof
Using branching formulae as in \cite[\S\S7--8]{BvdG}, one proves that the restriction of $\V_{2,0}$ to $\Sym^2\ab1\subset\ab2$ coincides with the symmetrization of $\V_2\times\V_0$ on $\ab1\times\ab1$. Its cohomology with compact support is then $\Q(-1)$ in degree $3$ and trivial in all other degrees by e.g. \cite[Thm.~5.3]{G-M1n}. Analogously, one shows that the cohomology of $\Sym^2\ab1$ with coefficients in the restriction of the local system $\V_{1,1}$ is trivial. Then the claim follows from the Gysin long exact sequence associated with the inclusion $\Sym^2\ab1\subset\ab2$.
\qed

\begin{lem}\label{ab3-v110}
The cohomology with compact support of $\ab3$ in the local system $\V_{1,1,0}$ is non-trivial only in degree $5$ and possibly in degrees $8$ and $9$ and is given in these degrees by $\cohc[\V_{1,1,0}]{\ab3}5=\Q(-1)$ and $\cohc[\V_{1,1,0}]{\ab3}8\cong\cohc[\V_{1,1,0}]{\ab3}9=\Q(-4)^{\oplus\epsilon}$ with $\epsilon\in\{0,1\}$.
\end{lem}

\proof
Branching formulae yield that the cohomology with compact support of the restriction of $\V_{1,1,0}$ to  $\tau_2(\M 2)\times\ab1$ is equal to $\Q(-5)$ (coming from the local system $\V_{1,1}\otimes\V_0$ on $\M2\times\ab1$) in degree $8$, to $\Q(-1)$ in degree $5$ (coming from the local system $\V_0\otimes\V_0(-1)$) and is trivial in all other degrees. Moreover, the restriction of $\V_{1,1,0}$ to $\Sym^3\ab1$ is trivial, as is easy to prove if one looks at the cohomology of the restriction of the universal abelian variety over $\ab3$ to $\Sym^3\ab1$.

It remains to consider the Gysin long exact sequence associated with the closed inclusion $\ab3^\op{red}\subset \ab3$. The only differential which can possibly be non-trivial is
$$
\Q(-5)=\cohc[\V_{1,1,0}]{\ab3^\op{red}}8\longrightarrow\cohc[\V_{1,1,0}]{\M3}9=\Q(-5).
$$
From this the claim follows.
\qed

In the investigation of the cohomology with compact support of the locus $\beta_2^0$ of semi-abelic varieties of torus rank $2$ we also need to consider the cohomology with compact support of the weight $4$ local system $\V_{2,2}$ on $\ab2$. In the forthcoming preprint \cite{OTM24}, we will show that the cohomology of $\V_{2,2}$ vanishes in all degrees. 
 For our application, however, we do not need such a complete result. The following lemma suffices:

\begin{lem}\label{ab2-v22}
The cohomology with compact support of $\ab2$ with coefficients in the local system $\V_{2,2}$ is $0$ in all degrees different from $3,4$. Furthermore, for every weight $k$ there is an isomorphism $$\operatorname{Gr}^W_k(\cohc[\V_{2,2}]{\ab2}3)\cong\operatorname{Gr}^W_k(\cohc[\V_{2,2}]{\ab2}4)$$ 
between the graded pieces of the weight filtration.
\end{lem}

\proof
First, we prove that the result holds in the Grothendieck group of rational Hodge structures. This requires to prove that the Euler characteristic of $\cohc[\V_{2,2}]{\ab2}\pu$ in the Grothendieck group of rational Hodge structures vanishes. By branching formulae, the cohomology with compact support of the restriction of $\V_{2,2}$ to $\Sym^2\ab1$ is equal to the cohomology of the local system $\V_0\otimes\V_0(-2)$, which is equal to $\Q(-4)$ in degree $4$ and trivial otherwise. On the other hand, the Euler characteristic of $\cohc[\V_{2,2}]{\M2}\pu$ was proved in \cite[Theorem 11.6]{Bergstrom-hyper} to be equal to $-[\Q(-4)]$. Then the additivity of Euler characteristics ensures that the Euler characteristic of $\V_{2,2}$ on $\M2$ vanishes. This means that the Euler characteristic of each graded piece of the weight filtration on $\cohc[\V_{2,2}]{\ab2}\pu$ is $0$. 

More generally, the fact that $\M2$ and $\Sym^2\ab1$ are affine of dimension $3$ and $2$ respectively, combined with the Gysin long exact sequence associated to $\Sym^2\ab1\hookrightarrow\ab2$  implies that the cohomology of $\ab2$ with values in any local system is trivial in degree greater than $3$. Thus, by Poincar\'e duality, the cohomology with compact support of $\ab2$ can be non-trivial only in degree larger than or equal to $3$. Furthermore, for non-trivial irreducible local systems $H^0$ (and hence $H^6_c$) vanishes, whereas $H^1$ (and hence $H^5_c$) is always zero by the Raghunathan rigidity theorem \cite{Raghunathan}. This means that the cohomology with compact support of $\V_{2,2}$ on $\ab2$ can be non-zero only in degrees $3$ and $4$. The cohomology groups in these degrees are then isomorphic when passing to the associated graded pieces of the weight filtration as a consequence of the vanishing of the Euler characteristic in the Grothendieck group of Hodge structures.
\qed

\bibliographystyle{amsalpha}

\end{document}